\documentclass[preprint,review,12pt]{elsarticle}
\usepackage{amssymb}
\usepackage{amsmath}
\usepackage{amsfonts}
\usepackage[table]{xcolor}
\usepackage{amsthm}
\usepackage{algorithm}
\usepackage{algpseudocode}
\usepackage{setspace}
\usepackage{lineno}
\usepackage{url}

\journal{}
\newcommand{\y}{{\mathbf y}}
\newcommand{\K}{{\mathbf K}}

\newcommand{\f}{{\mathbf f}}
\newcommand{\x}{{\mathbf x}}

\newcommand{\bpsi}{{\boldsymbol{\psi}}}
\def\st{\text{s.t.}}

\newcommand{\z}{{\mathbf z}}
\newcommand{\w}{{\mathbf w}}

\newcommand{\Name}{{AURORA}} 

\theoremstyle{plain}
\newtheorem{thm}{Theorem}[section]
\newtheorem{prop}[thm]{Proposition}

\theoremstyle{definition}
\newtheorem{defn}{Definition}[section]

\begin{document}

\begin{frontmatter}



\title{An automatic $L_1$-based regularization method for the analysis of FFC dispersion profiles with quadrupolar peaks}

 \author[label1]{G. Landi}
 \author[label1]{G.V. Spinelli}
 \author[label1]{F. Zama$^*$}
 \author[label2]{D. Chillura Martino}
 \author[label3]{P. Conte}
 \author[label2]{P. Lo Meo}
 \author[label4]{V. Bortolotti}
 \affiliation[label1]{organization={Department of Mathematics, University of Bologna},
            country={Italy}}

\affiliation[label2]{organization={Department of Biological,  Chemical and Pharmaceutical Sciences and Technologies, University of Palermo},
             country={Italy}}

 \affiliation[label3]{organization={Department of Agricultural, Food and Forest Sciences, University of Palermo},
             country={Italy}}

 \affiliation[label4]{organization={Department of Civil, Chemical, Environmental, and Materials Engineering, University of Bologna},
             country={Italy}}



\begin{abstract}
Fast Field-Cycling Nuclear Magnetic Resonance relaxometry
is a non-de\-struc\-tive technique to investigate molecular dynamics and structure of systems having a wide range of applications such as environment, biology, and food. Besides a considerable amount of literature about modeling and  application of such technique in specific areas, an algorithmic approach
to the related parameter identification problem is still lacking. We believe that a robust algorithmic approach will allow a unified treatment of different samples in several application areas.
In this paper, we model the parameters identification problem as a constrained $L_1$-regularized non-linear least squares problem.
Following the  approach  proposed in [\emph{Analytical Chemistry} \textbf{2021} 93 (24)], 
the non-linear least squares term imposes data consistency by decomposing the acquired relaxation profiles into relaxation contributions associated with ${^1H}-{^1H}$ and ${^1H}-{^{14}N}$ dipole-dipole interactions. The data fitting and the $L_1$-based regularization terms are balanced by the so-called regularization parameter.

For the parameters identification, we propose an algorithm that computes, at each iteration, both the regularization parameter and the model parameters. In particular, the regularization parameter value is updated according to a Balancing Principle and the model parameters values are obtained by solving the corresponding $L_1$-regularized non-linear least squares problem by means of the non-linear Gauss-Seidel method. We analyse the convergence properties of the proposed algorithm and run extensive testing on synthetic and real data. A Matlab software, implementing the presented algorithm, is available upon request to the authors.
\end{abstract}




\end{frontmatter}



%
%
%
\section{Introduction} \label{Intro}
Fast Field-Cycling (FFC) Nuclear Magnetic Resonance (NMR) relaxometry is a non-de\-struc\-tive magnetic resonance technique which is particularly useful in revealing information on slow molecular dynamics, which can only be carried out at very low magnetic field strengths.
Standard NMR relaxation experiments are only performed in a relatively large fixed magnetic field that determines the resonance frequency of the molecules under investigation. Conversely, FFC-NMR relaxometry \cite{kimmich1997field,conte2021applications} provides relaxation studies in a remarkably wide frequency range from approximately 1 kHz to 120 MHz.
The FFC technique allows one to evaluate how the rate $R_1$ (also referred to as longitudinal relaxation rate)  
of a sample varies by changing the strength of an
applied magnetic field, so forming the NMR Dispersion  (NMRD) profiles. Therefore, FFC-NMR relaxometry measurements can detect the motion across a wide range of timescales (from millisecond to picoseconds) within an  experiment.
In addition, frequency-dependent relaxation studies have the exceptional potential to reveal the underlying mechanisms of molecular motion (not just its timescale).
Spin relaxation theory represents  relaxation rates  as linear combinations of spectral density functions (Fourier transform of the time correlation function) characterising the motional frequencies and their intensities present in the correlation function \cite{farrar2012pulse}.
However, complex spins dynamical interactions may occur such as the Quadrupole Relaxation Enhancement (QRE) due to their intramolecular magnetic dipolar coupling with quadrupole nuclei of arbitrary spins $S \geq 1$ \cite{fries2015simple,kruk2019dynamics}.
In the case of  
nitrogen-containing systems, for instance, the presence of QRE  is represented by local maxima or peaks of the $R_1$ profiles due to $^{1}H-{^{14}N}$ interactions.  The positions of the  peaks depend on the quadrupole parameters which are determined by the electric field gradient tensor at the $^{14}N$ position. Consequently, even subtle changes in the electronic structure around $^{14}N$  reflect in changes of the position and shape of the quadrupole peaks. The QRE is a very sensitive fingerprint of molecular arrangement which has a wide range of  applications ranging from environmental science \cite{jeoh2017two}, the study of ionic liquids, proteins \cite{kruk2019dynamics}
and food \cite{conte2021fast,ates2021field}.\\

Despite the consistent literature about the modeling of relaxation rate $R_1$ of protons fluids within a confined environment (see for instance \cite{korb2011nuclear, mitchell2013exploring, faux2017explicit}) and  applications of FFC-NMR (see for example \cite{kruk20201h} and references therein), the study of a computational framework for the automatization of the FFC-NMR analysis is still missing. To the authors' best knowledge, only P. Lo Meo et al. \cite{lo2021heuristic}  propose a computational approach where,
following the ``model-free'' approach introduced by \cite{halle1998model,halle2009physical}, the relaxation rate $R_1$ is represented as the sum of a constant term (offset) accounting for very ``fast'' molecular motion, a term describing proper $^1H-{^1H}$ relaxation as an integral function of the correlated time distribution function, and a non-linear term depending on several  characteristic parameters related to the QRE occurence.

Therefore, the analysis of the NMRD profiles requires the solution of a parameter identification problem dealing with the estimation of the offset term, the correlation time distribution and the QRE parameters. In the present contribution, we formulate the parameter identification problem as a
regularized non-linear least squares problem with box constraints
and we propose a completely automatic strategy for its solution. 
In particular, the objective function
contains a non-linear least squares term, imposing data consistency, and a $L_1$-based regularization term, accounting for the known sparsity of the correlation time distribution function. These terms are balanced by the so-called regularization parameter. 
Physical constraints on the unknown parameters lead to bound constraints in the optimization problem.

The parameter identification problem crucially depends on the regularization parameter whose value has to be properly identified in order to perform a meaningful NMRD analysis.  
Therefore, our mathematical model depends on several parameters: the NMRD parameters (i.e. the offset, the correlation time distribution), the QRE parameters, and the regularization parameter.
The estimation of all these parameters is carried out by an iterative process where, at each iteration, the regularization parameter is computed according to a balancing principle \cite{Ito2011}. 
The NMRD and QRE parameters are estimated  solving the corresponding constrained optimization problem  by the constrained two-blocks non-linear Gauss-Seidel (GS) method \cite{GrippoSciandrone1999,GrippoSciandrone2000}, since the unknown  NMRD and QRE parameters can be naturally partitioned into two blocks. In the GS method, the objective function is iteratively minimized with respect to the offset and the correlation time distribution while the QRE parameters are held fixed; then, fixed the updated values for the offset and the correlation time distribution, the objective is minimized with respect to the QRE parameters. The first subproblem 
involves solving a  constrained linear least squares problem, obtained by the  {\em model-free} approach \cite{lo2021heuristic}, with an $L_1$ regularization term.
The second subproblem requires the solution of  a constrained non-linear least squares problem.

This computational approach, separating the contribution due to the offset and the relaxation distributions from the parameters of the quadrupolar relaxation, is able to provide a very accurate fit  not only of the overall NMRD profile, but also of the local maxima  due to the QRE.

Besides analysing the convergence of the proposed approach, we tested it on synthetic and real data aiming to illustrate the algorithm efficiency and its robustness to data noise.\\

The main contributions of the present paper can be  summarized as follows.
\begin{itemize}
\item  We formulate the problem of identifying the offset, the correlation time distribution and the QRE parameters from the NMRD profiles as a $L_1$-regularized non-linear least squares problem with box constraints related to physical properties of the parameters.
\item We derive an automatic procedure, named \Name \ 
(AUtomatic $L_1$-Regularized mOdel fRee Analysis) for the identification of \emph{all} the parameters of mathematical model, i.e, the NMRD parameters (the offset term, the correlation time distribution),  the QRE parameters and the regularization parameter and we analyse its convergence properties.
\item We prove the robustness of the proposed approach to data noise by testing it on synthetic and real NMRD profiles.
\end{itemize}
The remainder of this paper is organised as follows. In section \ref{mathmodel} we present the parameter identification problem; in section \ref{method} we introduce the solution method, analyse its properties and present the \Name\ algorithm.
The results from several numerical experiments are reported and discussed in section \ref{numres}. Finally, in section \ref{conclusions},
we draw some conclusions.

%
%
\section{The parameter identification problem} \label{mathmodel}
In the following, we first describe the continuous model for NMRD profiles, then we derive its discretization and, finally, we present the parameter identification problem.

\subsection{The continuous model of NMRD profiles}
Following the {\em model-free} approach, \cite{lo2021heuristic} proposes a model for the NMRD profiles $R_1$ made of three components:
\begin{equation}\label{eq:R1}
R_1(\omega)=R_0+R^{HH}(\omega)+R^{NH}(\omega)
\end{equation}
where $R_0$ is nonnegative offset  keeping into account very {\em fast} molecular motions, the term $R^{HH}(\omega)$ describes the 
correlation  distribution function $f(\tau)$ as: 
\begin{equation}
R^{HH}(\omega)=\int_0^{\infty}\left [\frac{\tau}{(1+(\omega \tau)^2)}+
\frac{4\tau}{(1+4(\omega \tau)^2)} \right ] f(\tau) \ d \tau
\label{eq:RHH}
\end{equation}
where $\tau$ is the correlation time, i.e., the average time required by a molecule to rotate one radiant or to move for a distance as large as its radius of gyration.
The term $R^{HN}(\omega)$ describes the occurrence of the {\em quadrupolar peaks} \cite{kruk20201h}:
\begin{multline}\label{eq:RNH}
  R^{HN}(\omega) =  C^{HN}\left(
             \begin{array}{ccc}
               \frac{1}{3}+\sin^2(\Theta)\cos^2(\Phi), & \frac{1}{3}+\sin^2(\Theta)\sin^2(\Phi), & \frac{1}{3}+\cos^2(\Theta) \\
             \end{array}
           \right) \cdot \\
           \left(
             \begin{array}{c}
               \displaystyle{\frac{\tau_Q}{1+(\omega-\omega_{-})^2\tau_Q^2} + \frac{\tau_Q}{1+(\omega+\omega_{-})^2\tau_Q^2}}\\
               \displaystyle{\frac{\tau_Q}{1+(\omega-\omega_{+})^2\tau_Q^2} + \frac{\tau_Q}{1+(\omega+\omega_{+})^2\tau_Q^2}} \\
               \displaystyle{\frac{\tau_Q}{1+(\omega-(\omega_{+}-\omega_{-}))^2\tau_Q^2} + \frac{\tau_Q}{1+(\omega+(\omega_{+}-\omega_{-}))^2\tau_Q^2}} \\
             \end{array}
           \right)
\end{multline}
where
\begin{itemize}
\item[i)] $C^{HN}$ refers to ``the gyromagnetic ratios and the average interaction distance of the nuclei'';
\item[ii)]  $\Theta$ and $\Phi$ are two angles accounting for  `` the orientation of the $^{1}H-^{14}N$ dipole-dipole axis with respect to
the principal axis system of the electric field gradient at the position of $^{14}N$ '';
\item[iii)] $\tau_Q$  is the correlation time for the $^1H-{^{14}N}$ quadrupolar interaction;
\item[iv)] $\omega_-$ and $\omega_+$ are the angular frequency position of the peaks on the NMRD profiles.
\end{itemize}
(We remark that in \eqref{eq:RNH} the $\cdot$ operator 
denotes the scalar product of two vectors.) 
\subsection{The discrete model for NMRD profiles}\label{DM}
Before describing the discretization of the continuous model \eqref{eq:R1}, let us introduce the following notation. Let $\boldsymbol{\omega}\in \mathbb{R}^m$ be the vector of the
$m$ Larmor angular frequency values ($\omega=2\pi\nu$, $\nu$ in Mhz) at which $R_1$ is evaluated, and let $\y\in \mathbb{R}^m$ be the corresponding observations vector, i.e., $y_i=R_1(\omega_i)$, \quad $i=1,\ldots,m$. Let $\f\in \mathbb{R}^n$ be the vector obtained by sampling $f(\tau)$  in $n$ logaritmically equispaced values $\tau_1, \ldots,  \tau_n$. 
Finally, let $\boldsymbol{\psi} \in \mathbb{R}^6$ be such that  $\psi_1\equiv C^{HN}$ , $\psi_2\equiv \sin^2(\Theta)$,  $\psi_3\equiv\sin^2(\Phi)$, $\psi_4 \equiv  \tau_Q$, $\psi_5 \equiv \omega_-$,  $\psi_6 \equiv \omega_+$.
The discrete model, obtained by discretizing the equations \eqref{eq:RHH} and \eqref{eq:RNH}, is
\begin{equation}\label{eq:invpb}
  \y=\mathcal{F}(\f,\boldsymbol{\psi},R_0) \equiv \mathcal{F}_1(\f)+ \mathcal{F}_2(\boldsymbol{\psi}) + R_0
\end{equation}
where $\mathcal{F}: \mathbb{R}^{n+6+1} \rightarrow \mathbb{R}^m$. The first term $\mathcal{F}_1: \mathbb{R}^{n} \rightarrow \mathbb{R}^m$, only depending on $\f$, is a linear function of $\f$ deriving from the discretization of the integral for $R^{HH}$ in \eqref{eq:RHH}:
\begin{equation}\label{}
  \mathcal{F}_1(\f) \equiv \K \f, \ \hbox{with } \K\in \mathbb{R}^{m \times n}, \ \   \f \in \mathbb{R}^{n}
\end{equation}
and
\begin{equation*}
  \K_{i,j}=\frac{\tau_j}{(1+(\omega_i \tau_j)^2)}+\frac{4\tau_j}{(1+4(\omega_i \tau_j)^2)}, \quad i=1,\ldots,m, \; j=1,\ldots,n.
\end{equation*}
In a typical FFC-NMR experiment, $m \ll n$.

The second term $\mathcal{F}_2(\boldsymbol{\psi}): \mathbb{R}^{6} \rightarrow \mathbb{R}^m$ represents the quadrupolar component $R^{HN}$ \eqref{eq:RNH} and only depends on the parameters $\psi_j$, $j=1,\ldots,6$:
\begin{multline}\label{}
   \left( \mathcal{F}_2(\boldsymbol{\psi}) \right)_i
    = \psi_1
    \left(
      \begin{array}{ccc}
        \frac{1}{3}+\psi_2(1-\psi_3), & \frac{1}{3}+\psi_2 \cdot \psi_3, & \frac{1}{3}+(1-\psi_2) \\
      \end{array}
    \right)
    \cdot \\
    \left(
      \begin{array}{c}
          \displaystyle{\frac{\psi_4}{1+( \omega_i-\psi_5)^2\psi_4^2}+\frac{\psi_4}{1+( \omega_i+\psi_5)^2\psi_4^2}}\\
          \displaystyle{\frac{\psi_4}{1+( \omega_i-\psi_6)^2\psi_4^2}+\frac{\psi_4}{1+( \omega_i+\psi_6)^2\psi_4^2}} \\
          \displaystyle{\frac{\psi_4}{1+( \omega_i-(\psi_6-\psi_5))^2\psi_4^2}+\frac{\psi_4}{1+( \omega_i+(\psi_6-\psi_5))^2\psi_4^2)}}
      \end{array}
    \right)
\end{multline}
for $i=1,\ldots,m$.

The last term in $\mathcal{F}$ is the constant parameter $R_0 \geq 0 $ representing the offset  in the NMRD curve.
\subsection{The parameter identification problem }\label{par:IP}
Mathematically, the problem of identifying the parameters $\f$, $\bpsi$ and $R_0$ from the observations $\y$ is an ill-conditioned non linear inverse problem \eqref{eq:invpb}. In order to stabilize the parameter identification procedure, we use a regularization approach adding some \emph{a priori} information on the unknown parameters. In particular, we use $L_1$ regu\-la\-ri\-zation to induce sparsity of $\f$ since the distribution $f(\tau)$ is known to be a sparse function with only a few non-null terms. Therefore, the parameter identification problem is reformulated as the following optimization problem
\begin{equation}\label{eq:IP}
    \begin{array}{ll}
    \min\limits_{\f, \boldsymbol{\psi}, R_0} & \displaystyle \| \y - (\mathcal{F}_1(\f)+ \mathcal{F}_2(\boldsymbol{\psi}) + R_0) \|_2^2 + \lambda \| \f \|_1 \\
    \st     & \f \geq \mathbf{0}, \\
            & \boldsymbol{\psi} \in \mathcal{B}_{\psi}, \\
            & R_0 \geq 0,
    \end{array}
\end{equation}
where the set $\mathcal{B}_{\psi}$ defines the box constraints on $\bpsi$:
\begin{equation}\label{eq:bounds}
  \mathcal{B}_{\psi} = \big \{ \boldsymbol{\psi } \, :  \,  \psi_1 \in [0,\bar{C}]; \ \psi_2, \psi_3 \in [0, 1]; \ \psi_4 \in [0,\bar{\tau}]; \
\psi_5, \psi_6 \in [ \omega_{\ell},  \omega_{u}] \big \}.
\end{equation}
The bounds on the parameters $\psi_i$, $i=1,\ldots,6$, can be derived from the physical properties of the system and from the data $\y$; a deeper discussion on this topic will be given in the Section \ref{numres}.

The regularization parameter $\lambda>0$  weights the contribution of the $L_1$ regularization term; the parameters  $(\f, \boldsymbol{\psi}, R_0)$  obtained by solving \eqref{eq:IP} depend critically on the value of $\lambda$.

%
%
\section{The solution method} \label{method}
The presented parameter identification method is an iterative procedure where, at each iteration, a value of the regularization parameter $\lambda$ is provided and the corresponding parameters
$(\f_\lambda, \boldsymbol{\psi}_\lambda, R_{0,_\lambda})$ are computed by solving problem \eqref{eq:IP}. The constrained two-blocks non-linear Gauss-Seidel (GS) method \cite{GrippoSciandrone1999,GrippoSciandrone2000} is used for its solution. In the following, we firstly describe the GS method and recall its convergence properties, then we introduce the iterative procedure for the regularization parameter computation, and, finally, we draw the overall parameter identification procedure.

\subsection{The constrained two-blocks Gauss-Seidel method}
In this subsection, we describe the GS method used for the solution of the constrained optimization problem \eqref{eq:IP}
for a fixed value of the regularization parameter $\lambda$.
To this end, we partition the unknowns of \eqref{eq:IP} into two blocks such that the data fitting term is linear with respect to the first block and non-linear with respect to the second block. Therefore, we reformulate problem \eqref{eq:IP} as follows:
\begin{equation}\label{eq:reformulatedIP}
    \begin{array}{ll}
    \min\limits_{\x_1, \x_2} & \displaystyle g(\x_1,\x_2) = \| \y - \K_e \x_1 - \mathcal{F}_2(\x_2)\|_2^2 + \lambda \| \x_1 \|_1 + \eta\|\x_1\|^2_2 \\
    \st     & \x_1 \in X_1, \\
            & \x_2 \in X_2,
    \end{array}
\end{equation}
where
\begin{gather}
  \x_1 \equiv (\f,R_0),  \quad  \x_2 \equiv \boldsymbol{\psi}  \\
  X_1=\{\x_1 \geq \mathbf{0}\}, \quad
  X_2\equiv \mathcal{B}_{\psi}
\end{gather}
and
\begin{equation}
  \ \K_e = \left[\begin{array}{c c}\K  & \mathbf{1}\end{array}\right]
   \in \mathbb{R}^{m\times (n+1)}.
\end{equation}
The last $L_2$-based penalty term $\eta\|\x_1\|^2_2$ in the objective function  has been introduced to ensure that $\K_e^T\K_e+\eta \mathbf{I}$ is a definite positive matrix; to this and, a small value for $\eta$, as $\eta=10^{-10}$ for example, can be fixed. Moreover, observe that in \eqref{eq:reformulatedIP}, the parameter $R_0$ has been included in the $L_1$-based penalty term.

The closed subsets $X_1\subseteq\mathbb{R}^{n+1}$ and $X_2\subseteq\mathbb{R}^6$ are both convex; the objective function $g(\x_1,\x_2)$ is continuous and it is convex with respect to $\x_1$ for fixed $\x_2$, but it is not convex with respect to $\x_2$ for fixed $\x_1$. However, since $\K_e^T\K_e+\eta \mathbf{I}$ is definite positive and $X_2$ is bounded, it is easy to show that $g$ is coercive on $X_1\times X_2$.
\begin{defn}
 A function $g:\mathbb{R}^q \rightarrow \mathbb{R}$ is called coercive in $X$ if, for every sequence $\{\x^{(k)}\} \in X$ such that $\|\x^{(k)}\|\rightarrow\infty$, we have
\begin{equation*}
  \lim_{k\rightarrow\infty} g(\x^{(k)}) = +\infty
\end{equation*}
\end{defn}
\begin{prop}
  The function $g:\mathbb{R}^{n+1+6} \rightarrow \mathbb{R}$ such that
  \begin{equation*}
      g(\x_1,\x_2) = \| \y - \K_e \x_1 - \mathcal{F}_2(\x_2)\|_2^2 + \lambda \| \x_1 \|_1 + \eta\|\x_1\|^2_2
  \end{equation*}
  is coercive in $X_1\times X_2$.
\end{prop}
\begin{proof}
   The function $g$ can be rewritten as
   \begin{equation*}
     g(\x_1,\x_2) = \x_1^T (\K_e^T\K_e + \eta \mathbf{I} )\x_1 + 2 \x_1^T \K_e^T (\mathcal{F}_2(\x_2)-\y) + \|\mathcal{F}_2(\x_2)-\y\|^2 +\lambda \|\x_1\|_1
   \end{equation*}
   where $\K_e^T\K_e + \eta \mathbf{I}$ is positive definite.
Let $\{(\x_1^{(k)},\x_2^{(k)})\}$ be a sequence in $X_1\times X_2$ such that $\lim_{k\rightarrow\infty} \|(\x_1^{(k)},\x_2^{(k)})\| = \infty.$ Since $X_2$ is bounded, we have
   \begin{equation}\label{eq:x1x2}
     \lim_{k\rightarrow\infty} \|\x_1^{(k)}\| = \infty \quad \text{and} \quad \lim_{k\rightarrow\infty} \|\x_2^{(k)}\| < \infty .
   \end{equation}
   Let $\mu>0$ be the smallest eigenvalue of $\K_e^T\K_e + \eta \mathbf{I}$. It holds
   \begin{align*}
     g(\x_1^{(k)},\x_2^{(k)})  & \geq  \mu \|\x_1^{(k)}\|^2 - 2 \|\K_e^T (\mathcal{F}_2(\x_2^{(k)})-\y)\|\|\x_1^{(k)}\| +\lambda \|\x_1^{(k)}\| + \\
     & \quad +\|\K_e^T (\mathcal{F}_2(\x_2^{(k)})-\y)\|^2 \\
                   & \geq \left( \mu\|\x_1^{(k)}\| - 2 \|\K_e^T (\mathcal{F}_2(\x_2^{(k)})-\y)\| + \lambda\right) \|\x_1^{(k)}\|
   \end{align*}
   From \eqref{eq:x1x2} we have $\mu\|\x_1^{(k)}\| - 2 \|\K_e^T (\mathcal{F}_2(\x_2^{(k)})-\y)\| + \lambda >0$ for sufficiently large $k$, then
   \begin{equation*}
  \lim_{k\rightarrow\infty} g(\x_1^{(k)},\x_2^{(k)}) = +\infty
\end{equation*}

\end{proof}
Continuity and coerciveness ensure the existence of at least one global minimizer of $g(\x_1,\x_2)$ in $X_1\times X_2$ \cite{ber99}.

In the constrained two-blocks Gauss-Seidel method, at each iteration, the objective function is minimized with respect to each of the block coordinate vectors $\x_i$ over the subsets $X_i$,  $i=1,2$, as summarized in Algorithm \ref{alg:GS}, 
where the  convergence condition is:
\begin{equation}
 |g(\x_1^{(k)},\x_2^{(k)} )-  g(\x_1^{(k-1)},\x_2^{(k-1)}) | \leq Tol_{GS} |g(\x_1^{(k)},\x_2^{(k)} )|.
 \label{eq:GS-conv}
 \end{equation}
%
%
\begin{algorithm}[h!]
\caption{Constrained two-blocks non-linear Gauss-Seidel method}\label{alg:GS}
\begin{algorithmic}[1]
\Function{GS}{$\x_1^{(0)},\x_2^{(0)}$}
\State Set $k=0$ and  $\x^{(0)}=(\x_1^{(0)},\x_2^{(0)})$.
\Repeat
\State $k=k+1$
\State Set  $\x_1^{(k)} \in \arg \min\limits_{\z \in X_1} \; g(\z,\x_2^{(k-1)})$
\State Set $\x_2^{(k)} \in \arg \min\limits_{\z \in X_2} \; g(\x_1^{(k)},\z)$
\Until{{\tt convergence condition} \eqref{eq:GS-conv}}
\State \textbf{return} $(\x_1^{(k)},\x_2^{(k)})$ 
\EndFunction
\end{algorithmic}
\end{algorithm}

%
%

  %
  %
  %
  %
%
%
We observe that the GS method is well defined since each subproblem  has solutions. Indeed, the function $g$ is strictly convex with respect of $\x_1$ and hence there exist at most one global minimum of $f$ over $X_1$ for fixed $\x_2$. On the other hand, Weierstrass's theorem guarantees the existence of at least one global minimum of $g$ over $X_2$ for fixed $\x_1$ since $g$ is continuous and $X_2$ is a closed and bounded set.

For general nonconvex, constrained problems, convergence of sequences generated by the GS method to critical points has been proved in \cite{GrippoSciandrone2000}. For the reader's convenience, we report here the main convergence result for the GS method and we refer to \cite{GrippoSciandrone2000} for its proof.
\begin{thm}
  Consider the problem
  \begin{equation}
    \begin{array}{ll}
    \min\limits_{\x_1, \x_2} & \displaystyle g(\x_1,\x_2) \\
    \st     & \x_1 \in X_1, \\
            & \x_2 \in X_2,
    \end{array}
\end{equation}
where $g$ is a continuously differentiable function and the subsets $X_i$ are closed, nonempty and convex for $i=1,2$.
Suppose that the sequence $\{(\x_1^{(k)},\x_2^{(k)})\}$ generated by the two-blocks GS method has limit points. Then, every limit point of $\{(\x_1^{(k)},\x_2^{(k)})\}$ is a critical point of the problem.
\end{thm}
We have already observed that the objective function $g$ of \eqref{eq:reformulatedIP} is coercive; since the level sets of continuous coercive functions are compact, the sequence
$\{(\x_1^{(k)},\x_2^{(k)})\}$ generated by the GS method has limit points (eventually, it has a convergent subsequence); hence, the GS method converges to critical points of \eqref{eq:reformulatedIP}.

We conclude this subsection with a remark on the solution of the two constrained subproblems to be solved at each iteration of algorithm \ref{alg:GS}.
The first subproblem at step 3 is a $L_1$-regularized least squares problem with nonnegativity constraints:
\begin{equation}
    \begin{array}{ll}
    \min\limits_{\z} & \displaystyle  \| \w-\K_e\z \|_2^2 + \lambda \sum_{i=1}^{m+1} z_i \\
    \st     & z_i\geq 0, \quad i=1,\ldots,m+1
    \end{array}
\end{equation}
where $\w = \y-\mathcal{F}_2(\x_2^{(k)})$.
For its solution, we use the truncated Newton interior-point method described in \cite{Kim2007}.

The second subproblem at step 4 is a bound constrained non-linear least squares problem:
\begin{equation}\label{eq:x2}
    \begin{array}{ll}
    \min\limits_{\z} & \displaystyle  \|\mathcal{F}_2(\boldsymbol{\z}) - \w \|^2 \\
    \st     & \z \in X_2
    \end{array}
\end{equation}
where $\w = \K_e\x_1^{(k+1)}-\y$. For its solution, we use the Newton Projection method \cite{Bersekas1982,Bersekas1984} where the Hessian matrix is approximated as in the Levenberg-Marquardt method \cite{NocedalWright06} since the Jacobian of $\mathcal{F}_2$ is ill-conditioned.


\subsection{Computation of the regularization parameter $\lambda$}
In order to correctly analyse the NMRD profiles, it is necessary to choose an appropriate value for the regularization parameter $\lambda$. Even if several parameter selection rules have been proposed in the literature for $L_2$-regularized minimization problems (see \cite{EnglBook2000,hansen1998,vogel02} for a theoretical discussion of such rules), the case of $L_1$-based regularization still remains largely unexplored. In \cite{Bonesky2008,clason2012}, the discrepancy principle has been investigated for nonsmooth regularization. This principle is difficult to be realized since it requires the prior knowledge of the noise norm and a solution of the discrepancy equation is not guaranteed to exist. In \cite{Ito2011}, the Balancing Principle (BP) has been proposed where the regularization parameter is selected by balancing, up to a multiplicative factor $\gamma$, the data fidelity and the regularization term, i.e,:
\begin{equation}\label{BP}
   \gamma  \lambda \| \x_1 \|_1 = \| \y - \K_e \x_1 - \mathcal{F}_2(\x_2)\|_2^2+\eta \|\x_1\|^2_2
\end{equation}
The regularization properties of the BP has been deeply investigated and a convergent fixed-point iterative scheme for its realization has been proposed in \cite{Ito2011}. We set $\gamma=1$ which gives the following rule for the regularization parameter selection:
\begin{equation*}
  \lambda = \displaystyle \frac{\| \y - \K_e \x_1 - \mathcal{F}_2(\x_2)\|_2^2+\eta \|\x_1\|^2_2}{\| \x_1 \|_1  }.
\end{equation*}

\subsection{The parameter identification method}
The proposed iterative method for the identification of both the NMRD parameters $\f$, $\bpsi$ and $R_0$ and the regularization parameter $\lambda$ is 
outlined in algorithm \ref{alg:PI} where, given an initial guess for $\lambda$, at each iteration, the NMRD parameters are computed by solving 
 problem \eqref{eq:IP} by the GS method, and the regularization parameter value is updated by the BP until the following convergence condition is met:
\begin{equation}
|\lambda^{(k+1)}- \lambda^{(k)} | \leq Tol_{\lambda} | \lambda^{(k)}|, \ \ 
Tol_{\lambda} >0.
\label{eq:lam-rule}
\end{equation}
We refer to this method as \Name \ (AUtomatic $L_1$-Regularized mOdel fRee Analysis). 

\begin{algorithm}[h!] \caption{\Name \  \label{alg:PI}}
\vskip 1mm
\begin{spacing}{1.15}
\begin{algorithmic}[1]
\State  Set $k=0$, $\eta=10^{-10}$ and choose a starting guess $\lambda^{(0)}$.
\Repeat
  \State $k=k+1$ 
  \State {\bf NMRD and QRE parameters update}
  $$ \hbox{By algorithm \ref{alg:GS} compute } (\x_2^{(k)},\x_2^{(k)})=GS(\x_1^{(k-1)},\x_2^{(k-1)}) \ \hbox{i.e.}$$ 
  $$
    (\x_1^{(k)},\x_2^{(k)}) \in \arg\min\limits_{\substack{\x_1\in X_1 \\ \x_2\in X_2}}
    \displaystyle  \; \| \y - \K_e \x_1 - \mathcal{F}_2(\x_2)\|_2^2 + \lambda^{(k)} \| \x_1 \|_1+\eta\|\x_2\|^2_2$$
  %
  \State \textbf{Regularization parameter update} $$\displaystyle \lambda^{(k+1)}=\frac{\| \y - \K_e \x_1^{(k)} - \mathcal{F}_2(\x_2^{(k)})\|_2^2+\eta\|\x_2^{(k)}\|^2_2}{\|\x_1^{(k)}\|_1}$$ 
\Until{{\tt convergence condition} \eqref{eq:lam-rule}}
\State \textbf{return} $(\f,R_0)=\x_1^{(k)}$ and $\bpsi = \x_2^{(k)}$ \Comment{Result $(\f,R_0,\bpsi)$}
\end{algorithmic}
\end{spacing}
\end{algorithm}
Following the analysis of the BP performed in \cite{Ito2011}, algorithm \Name\ can be viewed as a fixed point-like scheme for the problem  
\begin{equation}\label{}
     \begin{array}{l}
    (\x_1^*,\x_2^*) = \text{arg}\min\limits_{\substack{\x_1\in X_1 \\ \x_2\in X_2}}  \ \| \y - \K_e \x_1 - \mathcal{F}_2(\x_2)\|_2^2 + \lambda^* \| \x_1 \|_1 + \eta\|\x_1\|^2_2 , \\
    \lambda^* = \displaystyle \frac{\| \y - \K_e \x_1^* - \mathcal{F}_2(\x_2^*)\|_2^2+\eta \|\x_1^*\|^2_2}{\| \x_1^* \|_1  } .
    \end{array}
\end{equation}
The monotone convergence of the sequence $\{\lambda^{(k)}\}$ generated by the fixed point scheme has been proved in \cite{Ito2011} when $\lambda^{(0)}$ is chosen in an interval containing only one solution of equation \eqref{BP}. 
%
%
\section{Results and Discussion} \label{numres}
In this section, we present and discuss the results obtained  by a set of numerical experiments
to assess the accuracy,  robustness and efficiency of the 
proposed algorithm.
In paragraph  \ref{par:sett}, we describe the experimental setting. In paragraph \ref{par:robu}, we test \Name \ on a synthetic NMRD  profile  $R_1$
computed by  the model \eqref{eq:R1}
with assigned values of the  parameters $\boldsymbol{\psi}$,  $\f$ and $R_0$.  
We evaluate computational efficiency  and accuracy of \Name \ comparing it with some algorithms available in the Matlab optimization Toolbox.
Moreover, we investigate the algorithm robustness in presence of data noise.\\
Then,  in paragraph \ref{par:measured},  we report the results of the analysis of   NMRD profiles from two different samples: Dry Nanosponge (DN) and    Parmigiano-Reggiano (PR) cheese.

\subsection{Numerical Experimental setting}\label{par:sett}
All numerical computations are carried out using  Matlab R2021b on a laptop equipped with 
$2.9$ GHz Intel Core i7 quad-core processor and $16$ GB $2133$ MHz RAM. \\
For all tests,  the values  $\bar C$ and $\bar \tau$ in the constraints set $\mathcal{B}_{\psi}$ \eqref{eq:bounds} 
 are set equal to a value  large enough  so that the intermediate solutions 
 $\psi_1^{(k)}$ and $\psi_4^{(k)}$ never reach such bounds. 
 In our tests $\bar C=\bar \tau=100$  are suitable values.  
 %
The interval $[\omega_{\ell},  \omega_{u} ]$  in \eqref{eq:bounds},  representing the region where $R_1$  interrupts its decaying behaviour due to QRE,  is defined by inspection of the NMRD profile.   \\
 The starting guess for the parameter $\psi_1^{(0)} \equiv C^{HN}$ is obtained by the literature \cite{kruk2019dynamics}:
\begin{equation}
C^{HN} = \frac{2}{3} \left (\frac{\mu_0}{4 \pi} \frac{\gamma_H \gamma_N \hbar}{r_{NH}^3}\right )^2 \approx 0.18 \ \left [\frac{\mu s}{s^{2}} \right ] 
\label{eq:CHN}
\end{equation}
where the value of the physical constants is reported in table \ref{tab:const}.
\begin{table}[h!]
\centering
\begin{tabular}{ccc}
Constant & Description & Value \\
\hline
$\mu_0$ &  permeability of vacuum & $ 10^{-7} \ T^2 J^{-1} m^3 $\\
$\gamma_H$ & $^{1}H$ gyromagnetic factor& $2.577 \ 10^6 \ T^{-1}  s^{-1} $ \\
$ \gamma_N$ &  $^{14}N$ gyromagnetic factor &$3.078 \  10^6 \ T^{-1}  s^{-1} $\\
 $\hbar$ & reduced Planck's constant & $1.05472 \ 10^{-34} \ J \ s $ \\
 $r_{HN}$ & $^{1}H-^{14}N$ inter-spin distance &  $1.4 \ 10^{-10}  \ m$\\
 \hline
\end{tabular}
\caption{Characteristic constants for $C^{HN}$ in \eqref{eq:CHN}.}
\label{tab:const}
\end{table}
Concerning the quadrupolar parameters,  $\psi^{(0)}_2\equiv \sin^2{\Theta^{(0)}}$,  $\psi^{(0)}_3\equiv \sin^2{\Phi^{(0)}}$, 
 the initial values are equal to the mean of the corresponding upper and lower bounds  in $\mathcal{B}_{\psi}$,  i.e.  $1/2$.
The initial value of $\psi_4^{(0)}\equiv \tau_Q$, 
is set to $1$,  while $\psi^{(0)}_5\equiv \omega_-^{(0)}$, and $\psi^{(0)}_6\equiv \omega_+^{(0)}$   
are defined  as follows:
$$
  \psi^{(0)}_5 =  \omega_{\ell}+\frac{1}{4} |\omega_{u}- \omega_{\ell}|, \ \ \
\psi^{(0)}_6 = \omega_{u}-\frac{1}{4} |\omega_{u}- \omega_{\ell}|. 
$$
%
The computed results are evaluated by  the Mean Squared Error (MSE) 
$$  \text{MSE}=\frac{\| R_1 - \mathcal{F}(\f, \boldsymbol{\psi},R_0) \|^2}{m}, \ \ 
$$
and the Parameter Relative Error (PRE):
\begin{equation}
\text{PRE}(x)= \frac{\| x^{exact}-x^{computed}\|^2}{\| x^{exact}\|^2}
\end{equation}
with $x$ representing either the vector $\f$ or the scalars  $R_0$,  $\psi_i$, $i=1, \ldots, 6$. \\
The components of the vector $\boldsymbol{\psi}$ are referenced by the name in the physical model \eqref{eq:RNH},
 according to the  
mapping introduced in section \ref{DM},  and reported in table \eqref{tab:map}  for convenience.
\begin{table}[h!]
\centering
\begin{tabular}{cccccc}
 $C^{HN}$  & $ \Phi$ & $\Theta $ & $\tau_Q$ & $\omega_-$  & $\omega_+$   \\
 \hline
 $\psi_1$ & $asin(\sqrt{\psi_2})$& $asin(\sqrt{\psi_3)})$ & $\psi_4$ & $\psi_5$ & $\psi_6$\\
\end{tabular}
\caption{Quadrupolar parameters mapping.}
\label{tab:map}
\end{table}
All the tests apply algorithm \ref{alg:PI} with $Tol_{\lambda}=10^{-2}$ in \eqref{eq:lam-rule} and algorithm \ref{alg:GS} with 
$Tol_{GS}=10^{-6}$ in \eqref{eq:GS-conv}.\\
The computational cost is evaluated in terms of number of execution time.

\subsection{Synthetic test Problem}\label{par:robu}
To investigate the properties of \Name, we first test it on the synthetic NMRD profile $R_1$ represented in figure \ref{fig:SD_R1}(a),  and  obtained by setting the parameters of model \eqref{eq:R1}
as in the second column of table \ref{tab:T0}, with  the distribution function $\f^*$  represented in red in figure \ref{fig:res}(a).  Throughout the  paragraph we use the frequencies $\nu$ instead of the angular frequencies $\omega$, i.e.  $\nu_-\equiv \omega_-/(2 \pi)$ and $\nu_+\equiv \omega_+/(2 \pi)$.
\begin{table}[h!]
\centering
\begin{tabular}{cccc}
\hline
  & reference & computed & PRE \\
 \hline
$R_0$& $3.69$ &  $3.6868$ &$7.0267 \ 10^{-4} $ \\
 $C^{HN}$ & $18.84$ & $18.8453$ & $6.1449 \ 10^{-5}$ \\
  $\tau_Q$ & $ 0.96$ &  $0.9554$ & $8.5033 \ 10^{-6}$\\
   $\Theta$ & $1.09$ &  $1.0901$ & $6.1449  \ 10^{-5} $\\
   $ \Phi$ & $0.57$  & $0.5696$ & $ 6.9199 \ 10^{-4} $\\
    $\nu_- $ & $2.15$ & $2.1502$ & $5.7363 \ 10^{-6} $\\
    $\nu_+$ & $2.87$ & $2.8696$ & $1.1316 \ 10^{-6}$\\
    \hline
 \end{tabular}
 \caption{Model parameters: reference (second column), \Name  \ computed values (third column) and PRE (fourth column).  }
\label{tab:T0}
 \end{table}   
\begin{figure}[h!]
\centering
\includegraphics[width=6cm,height=5cm]{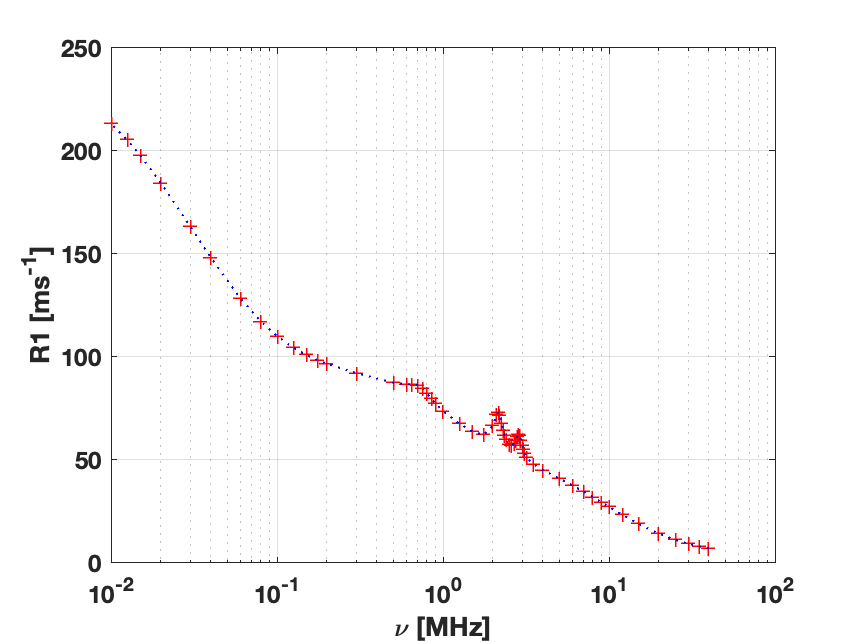}
\includegraphics[width=6cm,height=5cm]{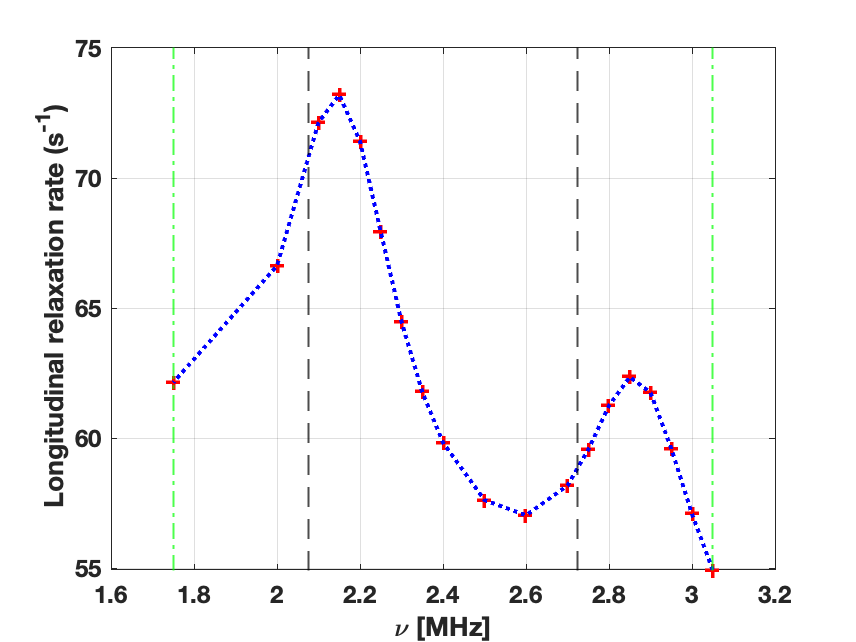}
\caption{Synthetic sample.  (a) Full NMRD profile (b) Zoom of NMRD profile in the reference interval $[\nu_{\ell},  \nu_{u}]$
represented by the left and right green vertical lines.  Left and right black vertical lines represent the values $ \psi^{(0)}_5/(2 \pi)$, $\psi^{(0)}_6/(2 \pi) $ respectively.}
\label{fig:SD_R1}
\end{figure}

The accuracy of the computed results can be appreciated  in the correlation distribution $\f$ and $R_1$ curves shown in figure \ref{fig:res}.
\begin{figure}[h!]
\centering
\includegraphics[width=6cm,height=5cm]{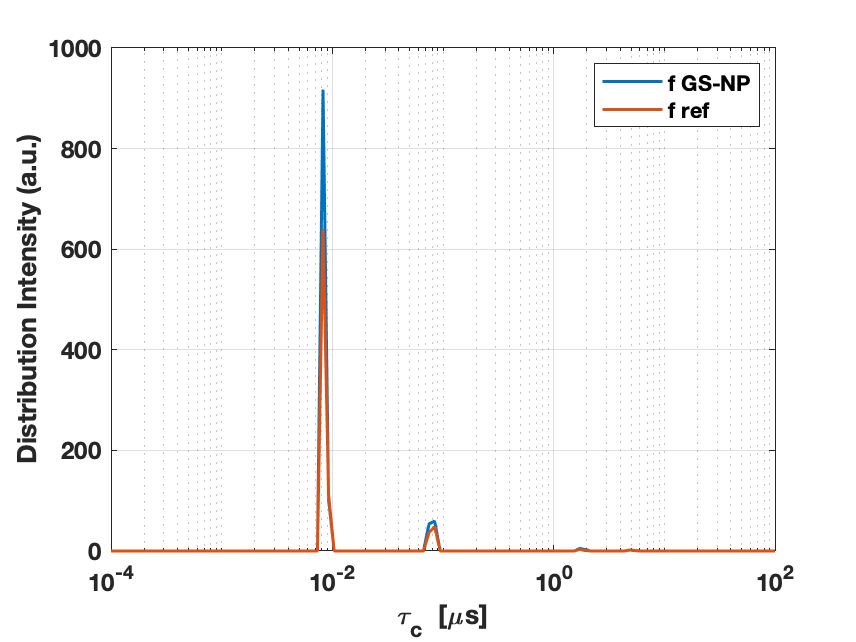}
\includegraphics[width=6cm,height=5cm]{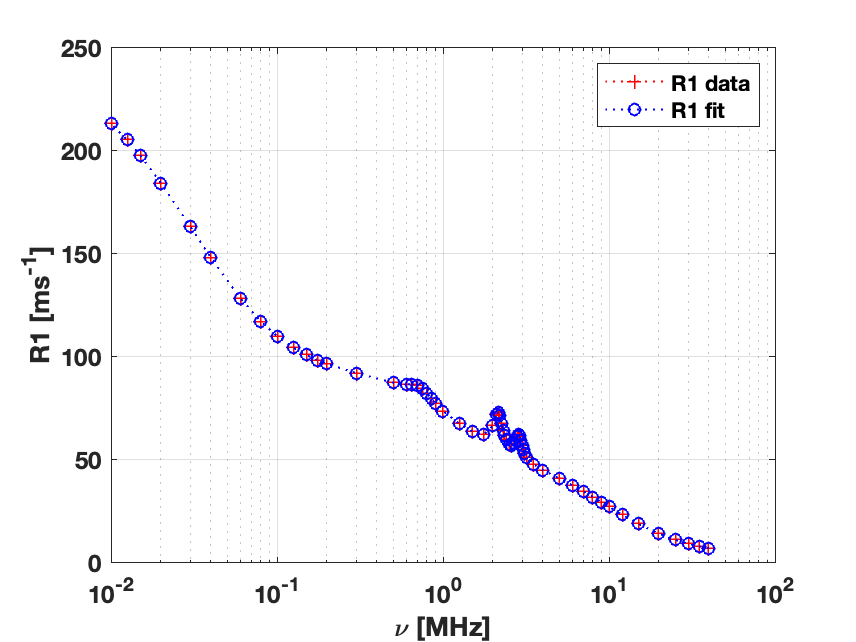}
\caption{Synthetic sample.  (a) Reference (red)  and computed correlation distribution (blue).  (b) Reference and computed $R_1$ profiles.}
\label{fig:res}
\end{figure}
To test the convergence behaviour  we evaluate the PRE and MSE  at each step of the GS method in algorithm \ref{alg:GS}.  Figure \ref{fig:Par_Ers}(a) shows the 
the behaviour of the relative errors for each parameter  $(\f,  R_0, C^{HN}, \Phi, \Theta, \tau_Q, \nu_-, \nu_+)$ compared to their reference values.
The convergence to reference parameters values is initially non monotonic for most parameters with the exception of $\tau_Q$  and $\nu_-$.  On the contrary,    MSE  has monotonic decrease as
 reported in figure \ref{fig:Par_Ers}(b).
\begin{figure}[h!]
\centering
\includegraphics[width=6cm,height=5cm]{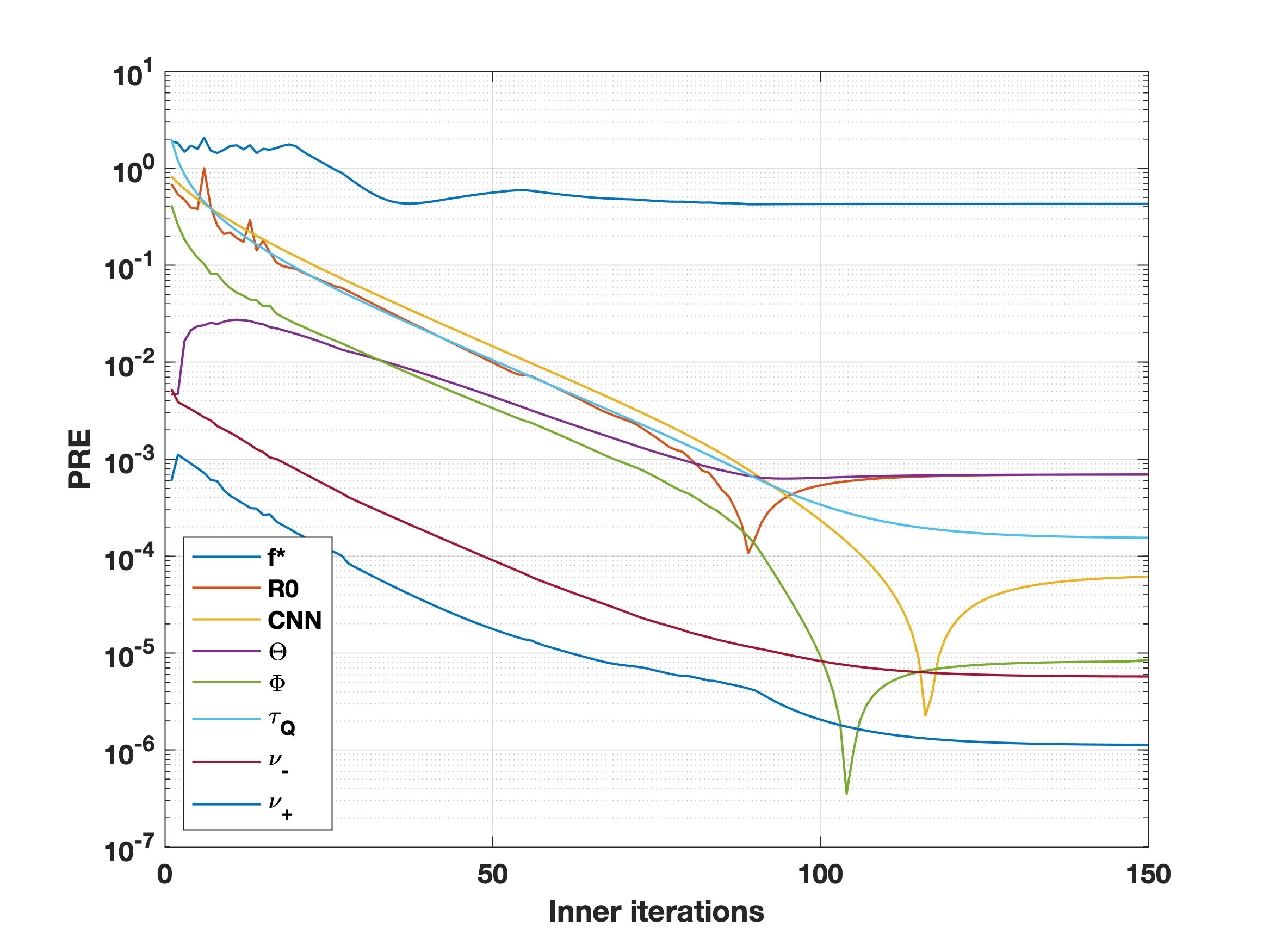}
\includegraphics[width=6cm,height=5cm]{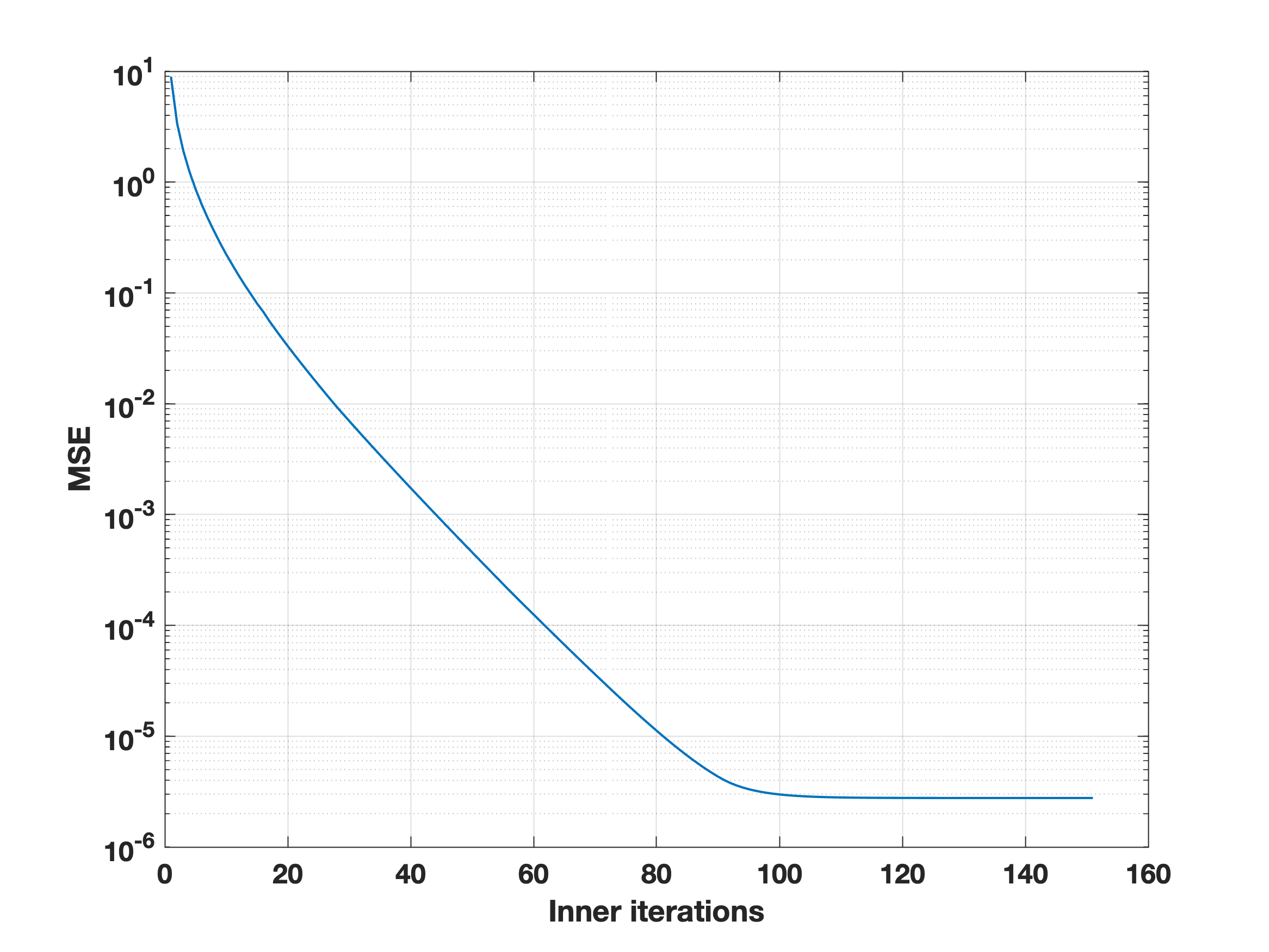}
\caption{Synthetic $R_1$.  (a) PRE values per iteration  (b)  MSE  values per iteration.}
\label{fig:Par_Ers}
\end{figure}
The  values of the computed parameters and relative errors  reported in the third and fourth columns of table \ref{tab:T0} confirm the excellent accuracy obtained by the proposed algorithm.  
The computed  value of the regularization parameter
is $\lambda^* = 1.216 \ 10^{-9}$   with   computation time of $115.15 \ s$.  \\
Although the convergence of the update  formula  \eqref{BP} depends on the initial guess $\lambda^{(0)}$, we experimentally found convergence for $\lambda^{(0)}$ in a    quite large interval ($[10^{-16},  10^0]$).  In figure \ref{fig:Lambda}  we represent the sequences $\lambda^{(k)}$,  $k=0, \ldots, 15$ obtained by algorithm \ref{alg:PI} with $\lambda^{(0)} \in \left \{ 10^{-16},   10^{-6},  10^{-4},  10^{-2}, 10^{0} \right \}$.  Optimal convergence ($k=1$) is obtained for $10^{-16} \leq  \lambda^{(0)} \leq 10^{-4}$ while  $\lambda^{(0)} > 10^{-4}$ causes a slight increase of the iterations number,  still preserving the convergence up to $\lambda^{(0)}=1$,  which is usually considered as a standard starting guess.  Therefore,  to keep computations efficient,  
$\lambda^{(0)} =10^{-6}$ is used throughout  the numerical experiments  of this section.
\begin{figure}[h!]
\centering
\includegraphics[width=6cm,height=5cm]{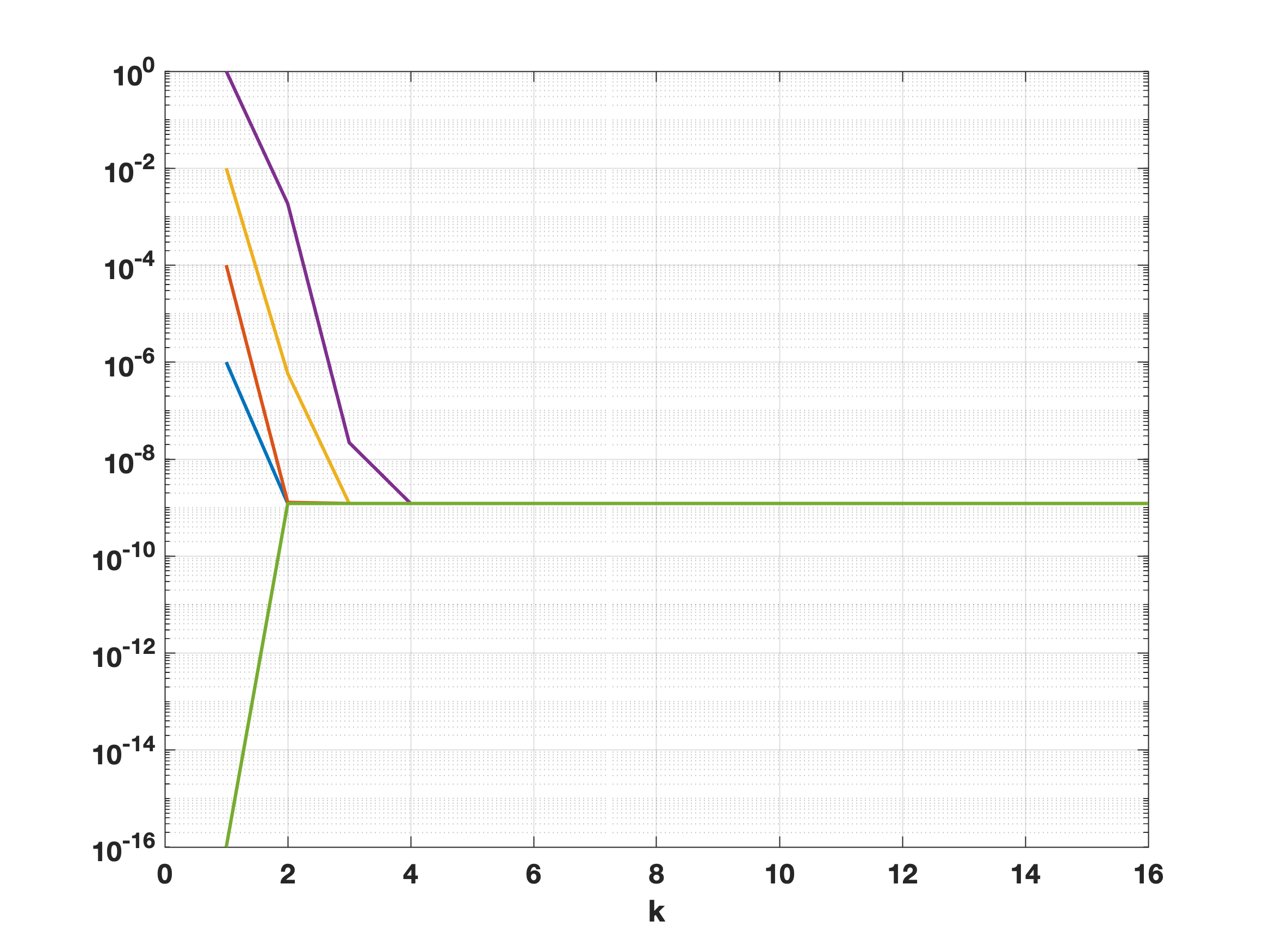}
\caption{Synthetic $R_1$.   Sequence $\{\lambda^{(k)}\}$,  obtained by \Name \  with $\lambda^{(0)}\in \left \{  10^{-16},   10^{-6},  10^{-4},  10^{-2}, 10^{0} \right \}$. }
\label{fig:Lambda}
\end{figure}

\paragraph{Comparison with Matlab solvers}
With this test problem, we aim to compare \Name \  with  several methods implemented by the Matlab function {\tt fmincon}:  such as interior-point  ({\tt ip}),  the active-set ({\tt as}),  the sequential quadratic programming ({\tt sqp}) and trust-region-reflective ({\tt trr}) methods.
We highlight that \Name \ automatically computes the value of the 
regularization parameter $\lambda$ while the Matlab function {\tt fmincon} solves the optimization problem \eqref{eq:IP}  for a fixed value of $\lambda$.
 Therefore,  we compare the GS  algorithm \ref{alg:GS}  with {\tt ip},  {\tt as},  {\tt sqp} {\tt trr} for the same fixed value $\lambda=1. \ 10^{-8}$, which we heuristically found to be a good value for all the methods.

Besides the automatic computation of the regularization parameter $\lambda$,  \Name \ splits the unknown parameters in two blocks and alternatively minimizes the objective function for $(R_0, \f)$, the offset and correlation distribution, and for the quadrupolar parameters $\boldsymbol{\psi}$.
Two different methods are used for the solution of the corresponding sub-problems. 
On the contrary,  {\tt fmincon} computes all the parameters applying the same method.  

%

Table \ref{tab:err_methods} shows the PRE  and MSE values (last row) obtained by  \Name \ (second column) and 
by the  Matlab solvers,  highlighting the smallest values. 
\begin{table}[h!]
\centering
\begin{tabular}{c|ccccc}
    \hline
    \hline
 & \multicolumn{5}{c}{PRE}\\
 Parameter  &  \Name &     {\tt ip} &    {\tt active-set}    &    {\tt sqp }     &     {\tt trr} \\
 \hline
  $\f$      &  $\mathbf{4.2834\ 10^{-1} }$ &  1.5509            &   1.4497            &  1.3020                & $8.5279\ 10^{-1}$      \\ 
  $R_0$     &  $\mathbf{7.0032\ 10^{-4}}$       &  $9.9629\ 10^{-1}$ &   1.0000            &  $2.7930\ 10^{-1}$     & $1.3671\ 10^{-1}$      \\ 
  $C^{HN}$  &  $5.8238\ 10^{-5}$           &  4.2908            &   $9.6353\ 10^{-1}$ &  $\mathbf{1.5045\ 10^{-5}}$ & $1.1591\ 10^{-2}$     \\  
  $\Theta$  &  $\mathbf{6.9108\ 10^{-4}}$       &  $6.5929\ 10^{-2}$ &   $7.7862\ 10^{-2}$ &  $7.2072\ 10^{-4}$     & $1.5758\ 10^{-2}$   \\    
  $\Phi$    &  $\mathbf{8.7093\ 10^{-6}}$       &  $5.5535\ 10^{-1}$ &   2.1372            &  $2.6548\ 10^{-5}$     & $7.2619\ 10^{-3}$ \\   
  $\tau_Q $ &  $\mathbf{1.5660\ 10^{-4}}$     &  $9.9228\ 10^{-1}$ &   $3.1584\ 10^{1}$  &  $1.8903\ 10^{-4}$     & $1.1033\ 10^{-2}$ \\      
  $\nu_-$   &  $5.7679\ 10^{-6}$                  &  $4.0856\ 10^{-1}$ &   $2.2756\ 10^{-1}$ &  $\mathbf{5.6228\ 10^{-6}}$ & $5.9438\ 10^{-5}$ \\   
  $\nu_+$   &  $\mathbf{1.1391\ 10^{-6}}$       &  $5.5197\ 10^{-2}$ &   $3.3362\ 10^{-2}$ &  $1.2084\ 10^{-6}$     & $1.8516\ 10^{-5}$ \\
\hline
\hline
  $MSE $       & $\mathbf{2.8131\ 10^{-6}}$     &  9.1906   &    9.0766  &   $3.1658\ 10^{-6}$     &     $2.8289\ 10^{-3}$    \\ 
\hline
\end{tabular}
\caption{Parameter relative errors and MSE of \Name \ and methods implemented by the Matlab function {\tt fmincon}. }
\label{tab:err_methods}
\end{table}
The  distribution ${\f}$ computed by  {\tt sqp} is shown in figure \ref{fig:sqp}.
\begin{figure}[h!]
\centering
\includegraphics[width=6cm,height=5cm]{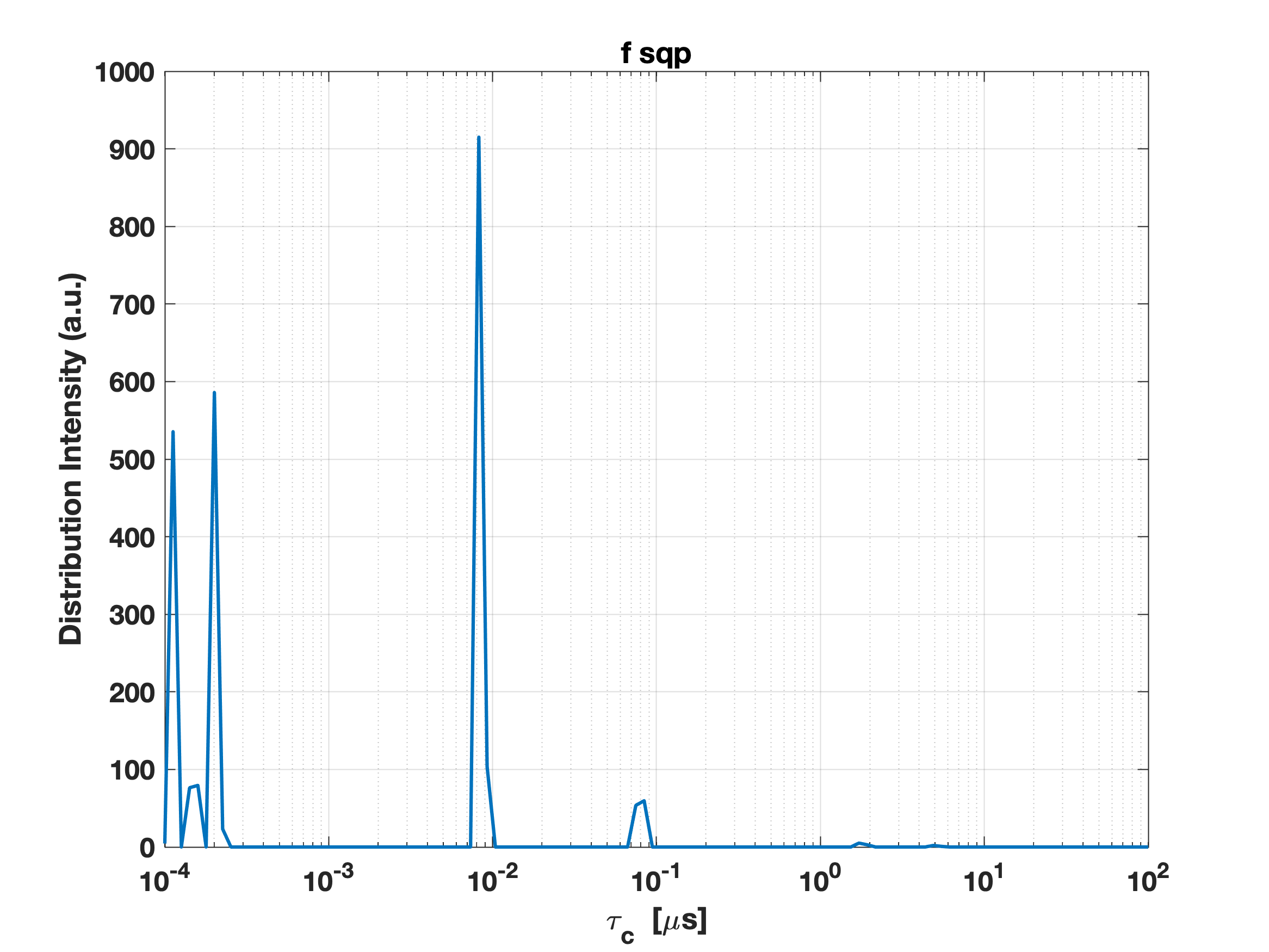}
\caption{Correlation time distribution $\f$ computed by {\tt sqp} method. }
\label{fig:sqp}
\end{figure}
We observe that  \Name \ has globally superior accuracy both in data fitting and parameter estimation.
Only {\tt sqp}  has   MSE value similar to 
\Name \ ($3.1658e-06$ compared to $2.8131e-06$),  and a slightly better PRE for parameters $C^{HN}$ and $\nu_-$, 
but  the amplitude distribution in figure \ref{fig:sqp} shows too many spurious peaks.
\paragraph{Test with noisy data}
In this paragraph we test the algorithm robustness to data perturbations
by computing noisy data $\y^{\delta}\in \mathbb{R}^m$  from a random uniformly distributed vector $\mathbf{v} \in \mathbb{R}^m$ with values in the interval $[-1, 1]$ s.t.
$$y^{\delta}_i = y_i(1 + \delta v_i), \ i=1, \ldots, m$$
and consider the cases $\delta =1\%, 5\%, 10\%$.
Computing $500$ noisy samples $\y^{\delta}_j$ we run \Name \ and compare the errors on the estimated parameters as well as reconstructed NMRD profiles. 

 For the noise values $\delta=1 \%, 5 \%,10\%$,  we compute 
 the  mean  PRE for each parameter and represent the mean  values  in the bar plot shown in figure \ref{fig:bars}
 together with the product $C^{HN} \cdot \tau_Q$.
 \begin{figure}[h!]
\centering
\includegraphics[width=8cm,height=6cm]{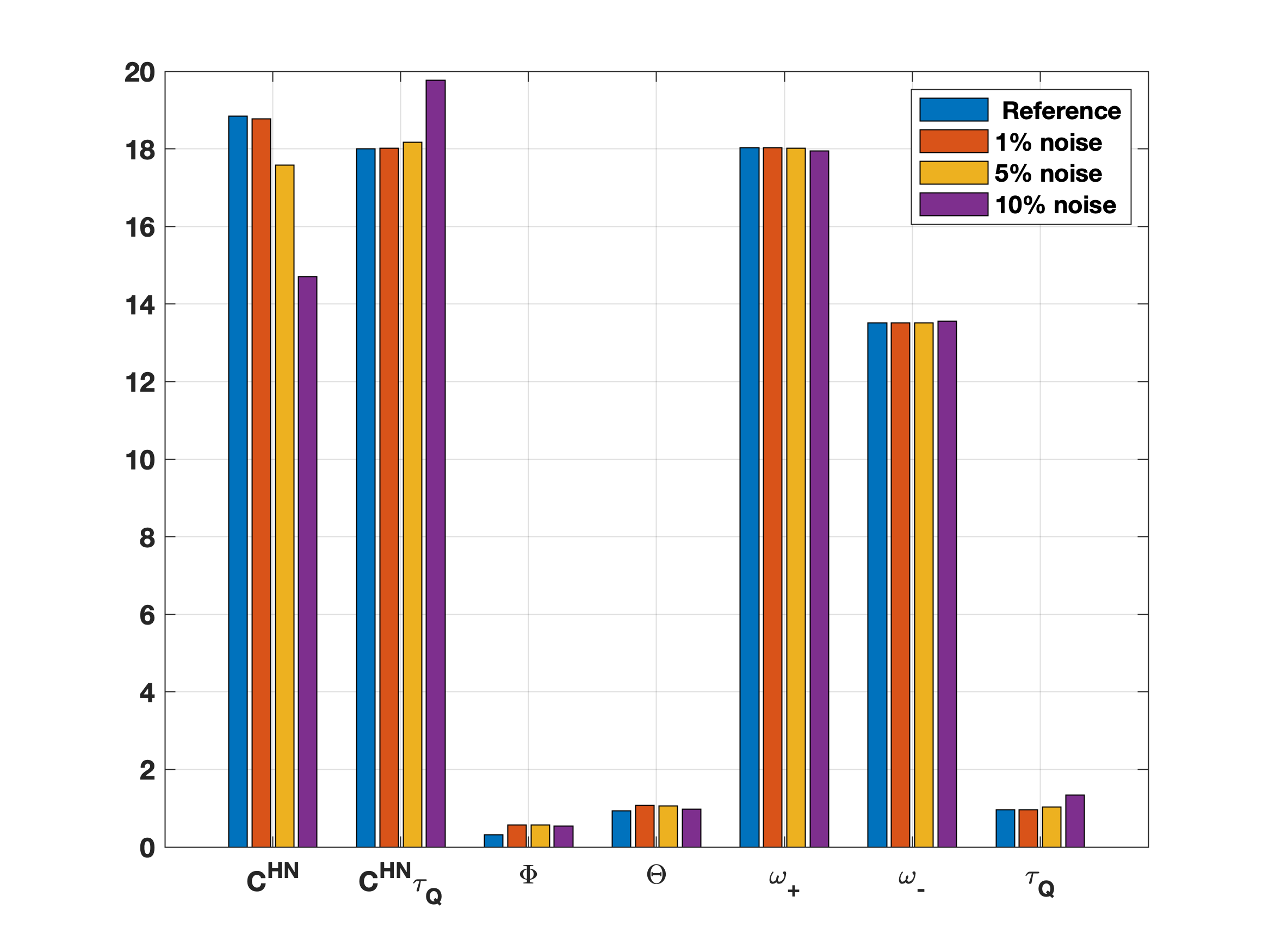}
\caption{Mean parameter values computed by 500 noisy NMRD profiles  with noise $\delta =1\%, 5\%10\%$.}
\label{fig:bars}
\end{figure}
The mean PRE and MSE are reported in table \ref{tab:T2}. 
 \begin{table}
 \centering
\begin{tabular}{c|ccc}
    \hline
    \hline
 & \multicolumn{3}{c}{PRE}\\
 \hline
  & $1 \%$ & $5 \%$ & $10 \%$ \\
 \hline
    $\f$       &   $5.9019\ 10^{-1}$  &  $1.1816        $  &  $1.4509         $ \\
    $R_0$      &   $3.6393\ 10^{-2}$  &  $1.6726\ 10^{-1}$  &  $1.8099\ 10^{-1}$ \\
    $C^{HN}$   &   $3.3625\ 10^{-2}$  &  $2.7021\ 10^{-1}$  &  $4.7742\ 10^{-1}$ \\
    $\Theta$   &   $2.3023\ 10^{-2}$  &  $1.0678\ 10^{-1}$  &  $2.1726\ 10^{-1}$ \\
    $\Phi$     &   $3.5151\ 10^{-2}$  &  $4.0280\ 10^{-1}$  &  $6.5910\ 10^{-1}$ \\
    $\tau_Q $  &   $4.4998\ 10^{-2}$  &  $1.8862         $  &  $1.1095\ 10^{1} $ \\
    $\nu_-$    &   $4.3917\ 10^{-3}$  &  $4.8712\ 10^{-2}$  &  $7.2441\ 10^{-2}$ \\
    $\nu_+$    &   $3.0889\ 10^{-3}$  &  $3.8712\ 10^{-2}$  &  $5.6856\ 10^{-2}$ \\
    \hline
    \hline
      MSE        &    $1.5980\ 10^{-1}$  & 3.1441 &  $1.0055\ 10^{1}$ \\
    \hline
\end{tabular}
 \caption{Mean PRE  and  MSE on 500 noisy NMRD profiles with  $\delta=1\%, 5\%, 10\%$. }
 \label{tab:T2}
 \end{table}
The computed $R_1$ curves and the zoom in the QRE interval are shown in figures \ref{fig:R1_500_01},\ref{fig:R1_500_05} and 
\ref{fig:R1_500_10} for $\delta =1 \%, 5 \%, 10 \%$ respectively.
\begin{figure}[h!]
\centering
\includegraphics[width=6cm,height=5cm]{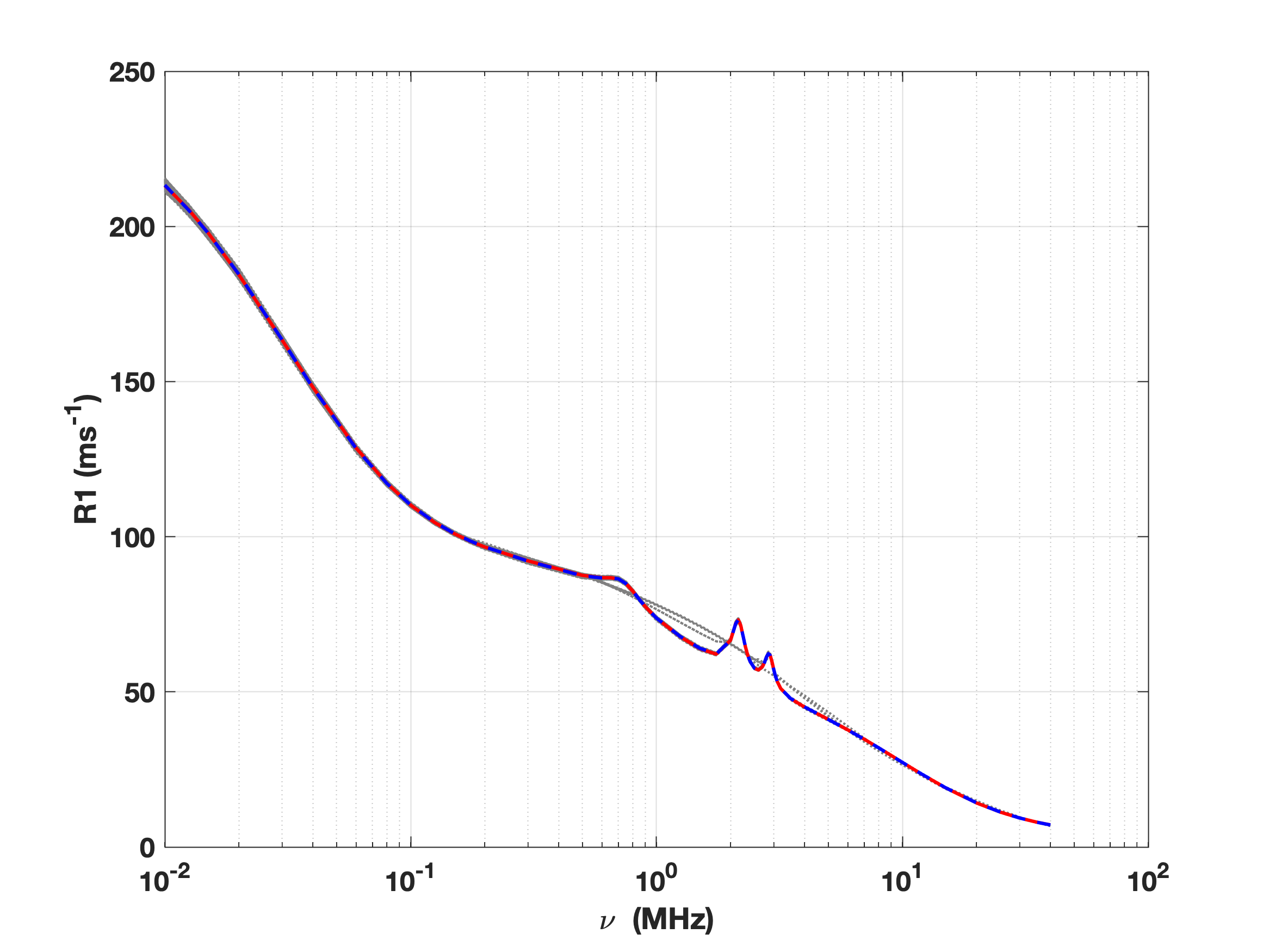}
\includegraphics[width=6cm,height=5cm]{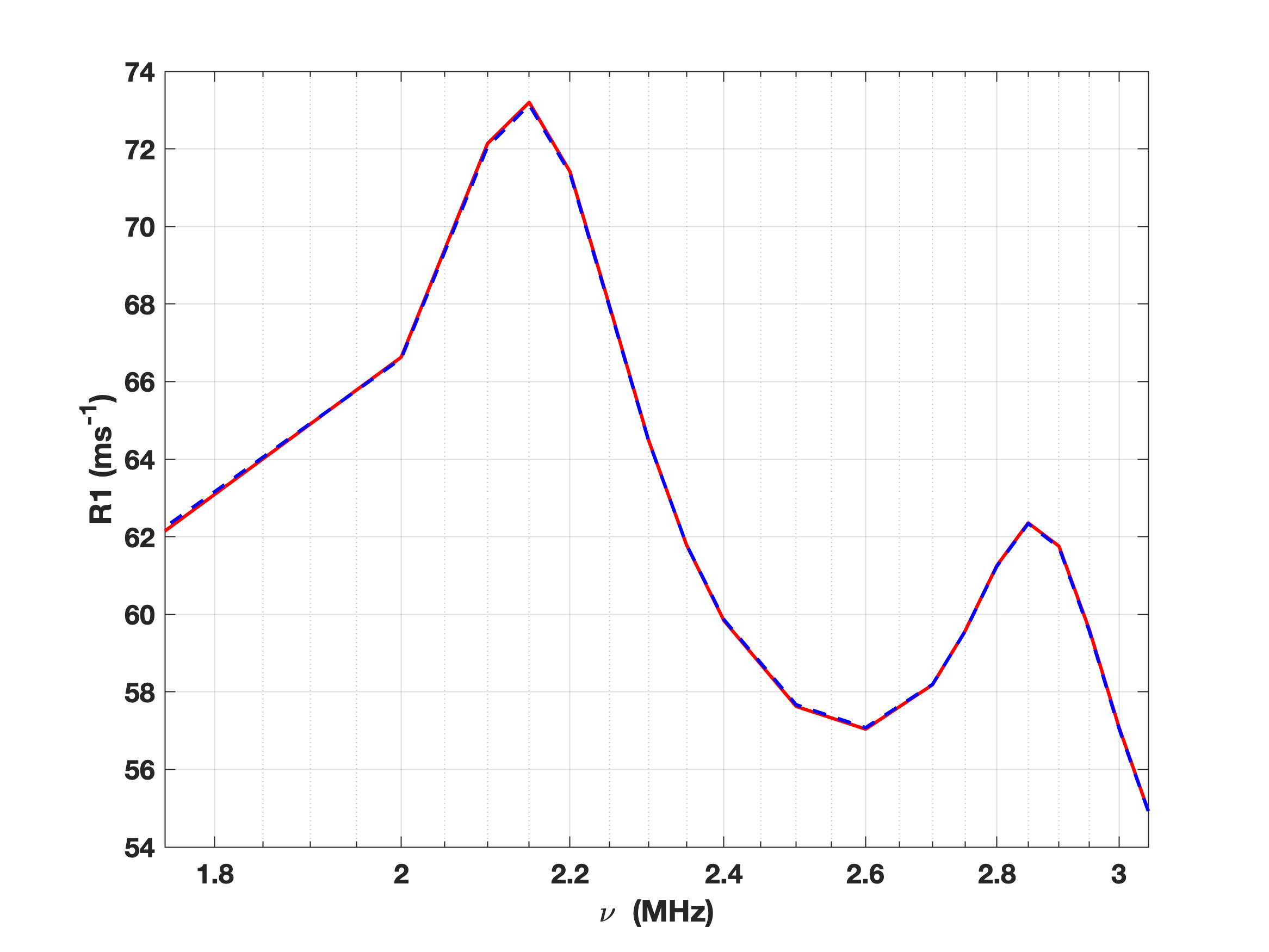}
\caption{Fit of NMRD obtained from 500 noisy Synthetic NMRD curves with noise $\delta=1\%$.  (a) Light gray: 500 fitted $R_1$ curves,  Red line: Reference NMRD curve.  Blue line: average over 500 fitted $R_1$ values.   (b) zoom in QRE interval.}
\label{fig:R1_500_01}
\end{figure}
\begin{figure}[h!]
\centering
\includegraphics[width=6cm,height=5cm]{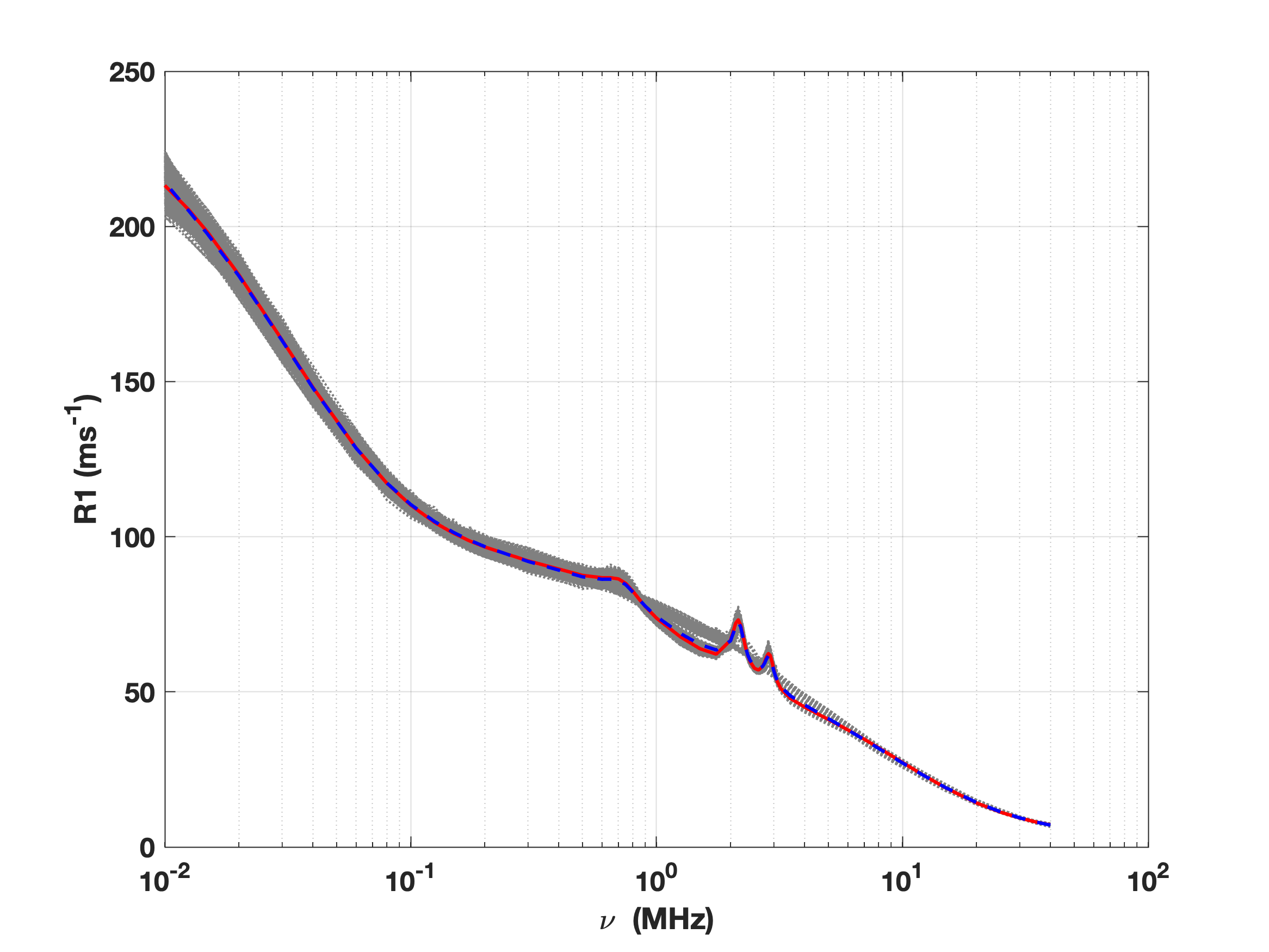}
\includegraphics[width=6cm,height=5cm]{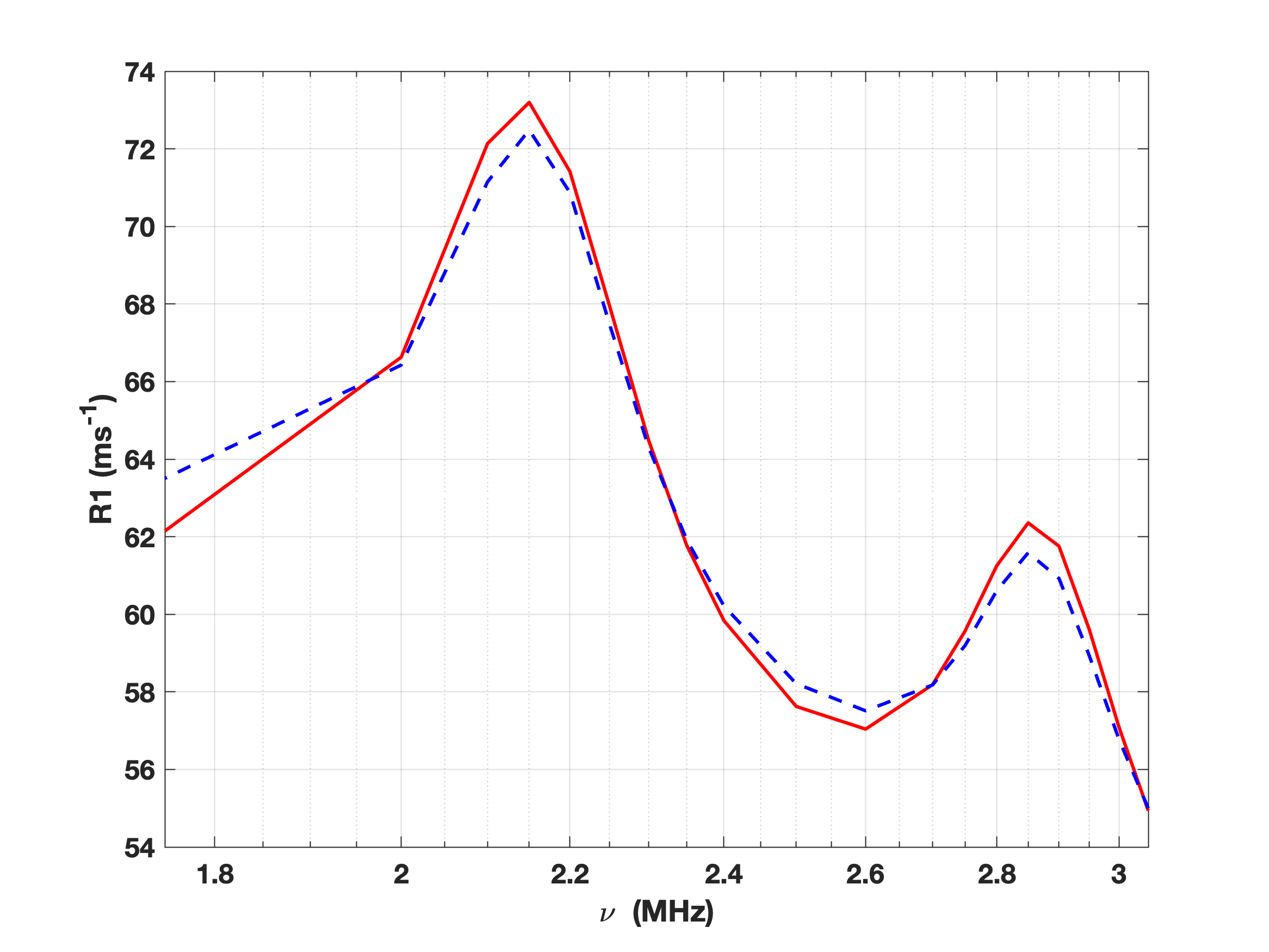}
\caption{Fit of NMRD obtained from 500 noisy Synthetic NMRD curves with noise $\delta=5\%$.  (a) Light gray: 500 fitted $R_1$ curves,  Red line: Reference NMRD curve.  Blue line: average over 500 fitted $R_1$ values.   (b) zoom in QRE interval.}
\label{fig:R1_500_05}
\end{figure}

\begin{figure}[h!]
\centering
\includegraphics[width=6cm,height=5cm]{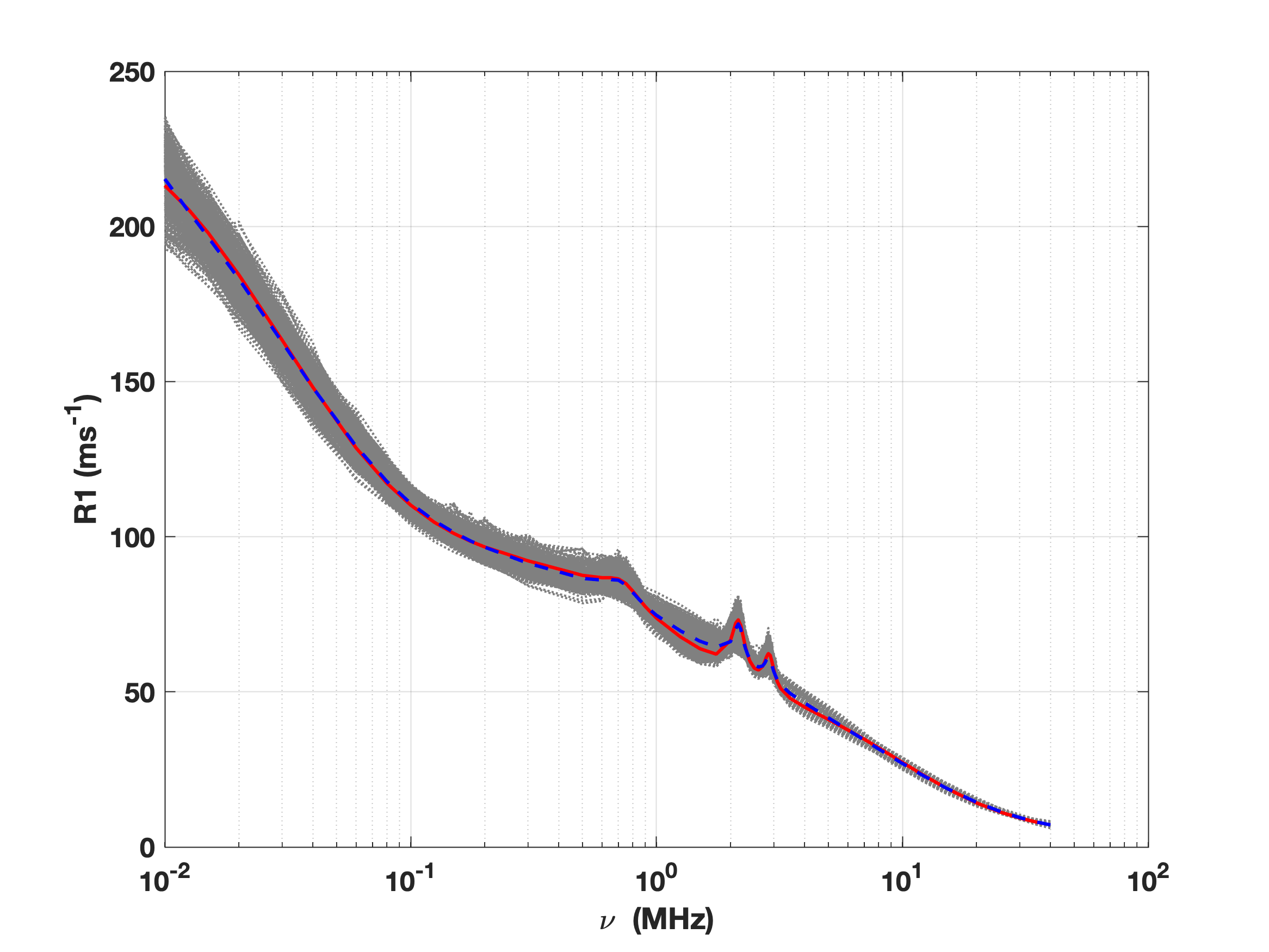}
\includegraphics[width=6cm,height=5cm]{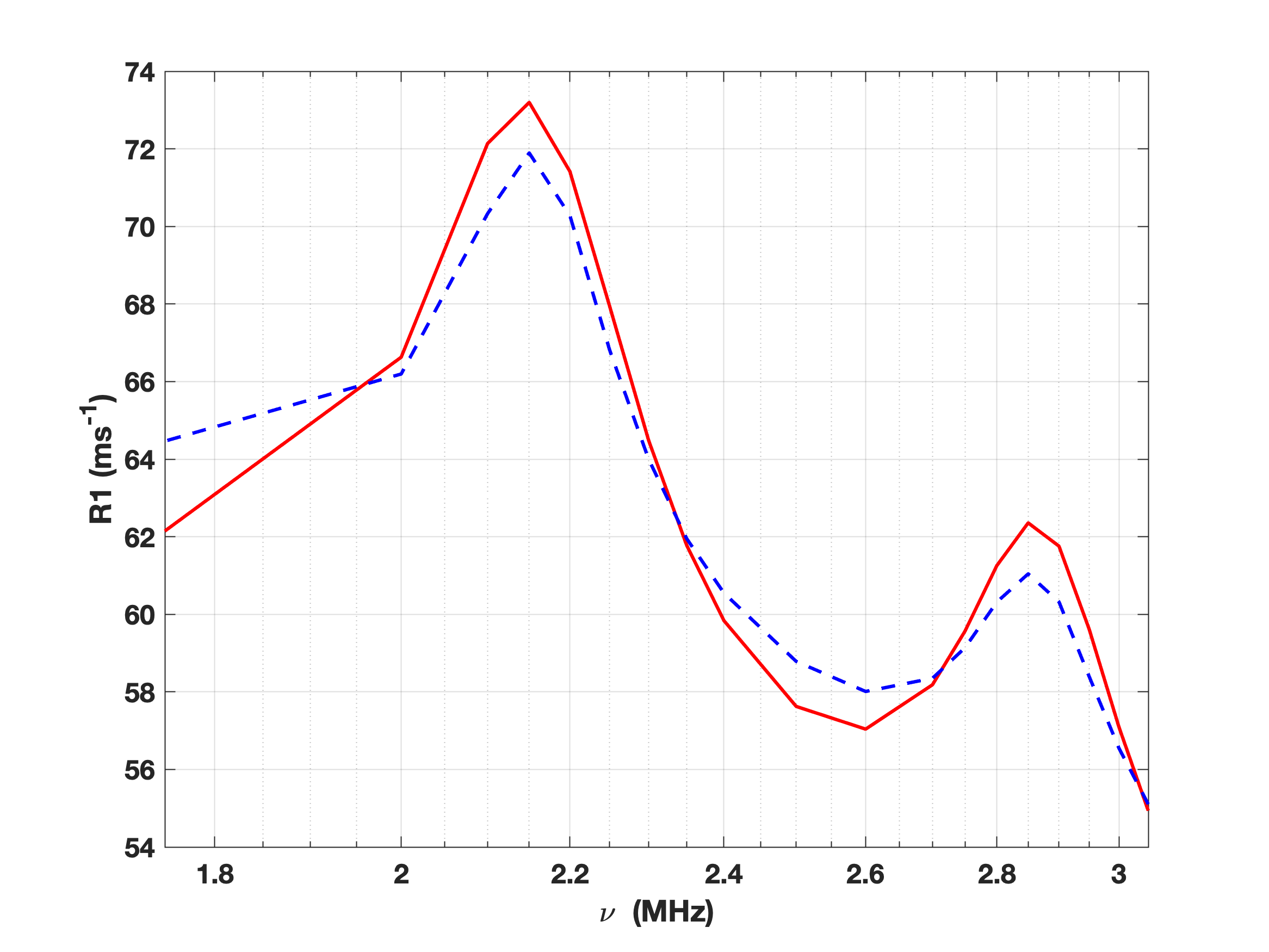}
\caption{Fit of $R_1$ obtained from 500 noisy synthetic NMRD profiles with noise $\delta=10\%$.  (a) Light gray: 500 fitted $R_1$ curves,  Red line: Reference NMRD curve.  Blue line: average over 500 fitted $R_1$ values.   (b) zoom in QRE interval.}
\label{fig:R1_500_10}
\end{figure}

 In figure \ref{fig:bars}, we observe that data noise affects mainly $C^{HN}$,  $\tau_Q$ and $\Phi$ values.  However, considering the value of the product $C^{HN} \tau_Q$,  represented by  the second group in figure \ref{fig:bars}, 
 we  see that the value is preserved when $\delta = 1\%,  5\%$.  This feature is a physical characteristic and allows us to consider   accurate the related parameters. 
 
 Although the average MSE  increase with data noise,  the computed average $R_1$ curves show a very good agreement to the reference NMRD profiles
(figures  \ref{fig:R1_500_01}, \ref{fig:R1_500_05} and 
\ref{fig:R1_500_10}).  The QRE is well reproduced even with high noise (figures \ref{fig:R1_500_01}(b), \ref{fig:R1_500_05}(b) and 
\ref{fig:R1_500_10}(b)).

\subsection{NMRD profiles from FFC measures}\label{par:measured}
In this paragraph we consider the NMRD profiles obtained from two different materials described in \cite{lo2021heuristic}.
\begin{itemize}
\item A sample of 24-month aged Parmigiano-Reggiano (PR) cheese.
The NMRD profile represented in figure \ref{fig:NMRD_samp}(a) has 
$m=48$ values  with  confidence intervals ranging from 
$\pm 0.35 \%$ to $\pm 3.07 \%$ of the  value.
The  quadrupolar  peaks,  represented in figure \ref{fig:R1_zoom}(a),  correspond to frequency  values $\nu_-=2.1$ and $\nu_+= 2.8$ of values $R1_-= 32.2 \ s^{-1}$ and $R1_+= 30.7 \ s^{-1}$ respectively.
\item Dry nanosponge (DN).  
In this case the NMRD profile represented in figure \ref{fig:NMRD_samp}(b) has 
$m=44$ values  with  confidence intervals ranging from 
$\pm 0.47 \%$ to $\pm 1.54 \%$ of the  value.
The quadrupolar peaks,  represented in figure  \ref{fig:NMRD_samp}(b),  correspond to frequency  values  $\nu_-=2.4991 \ MHz$ and $\nu_{+}=3.1488 \ MHz$ of values $R1_-= 104.85 \ s^{-1}$ and $R1_+= 104.85 \ s^{-1}$ respectively. 
\end{itemize}
\begin{figure}[h!]
\centering
\includegraphics[width=6cm,height=5cm]{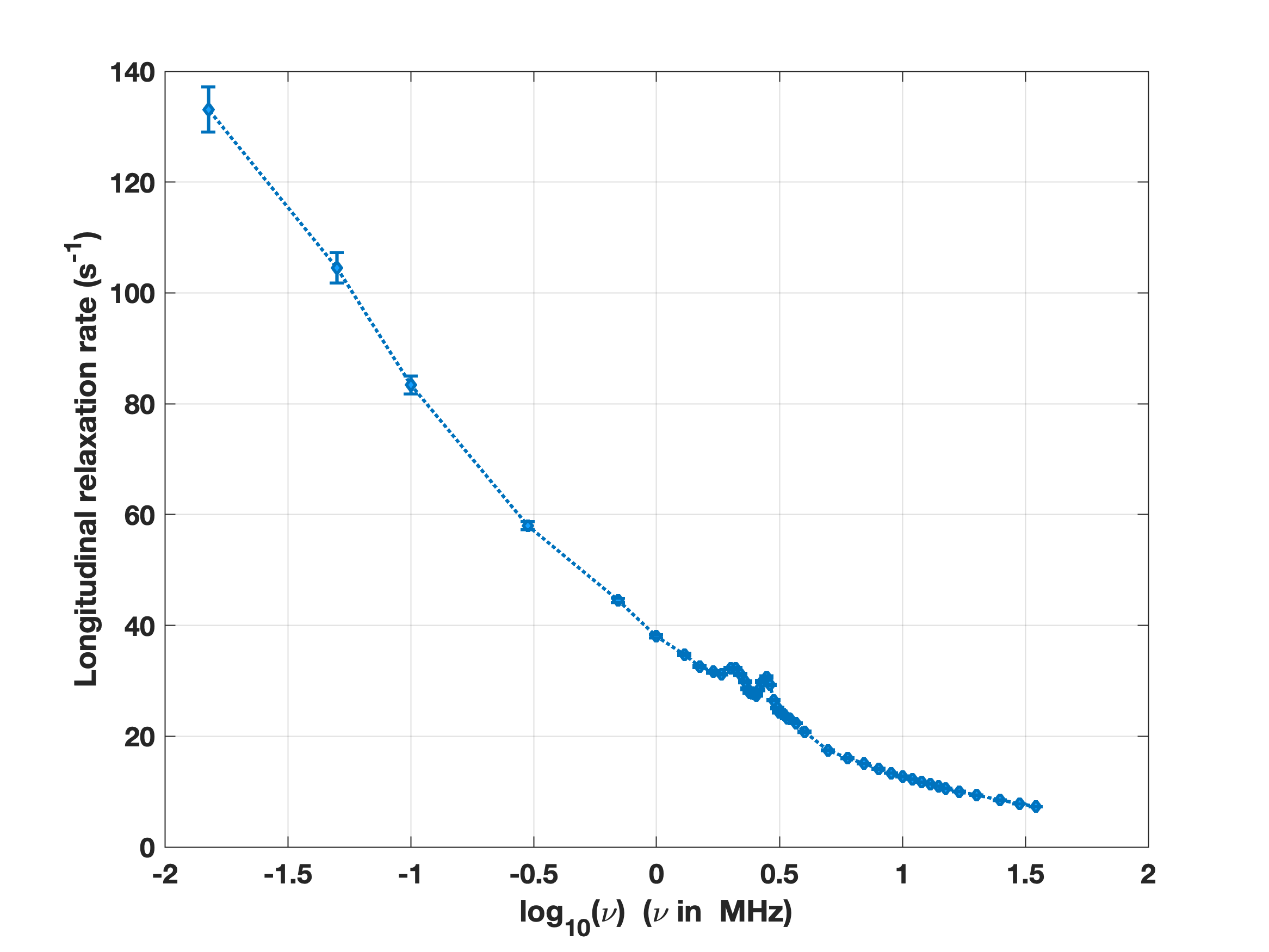}
\includegraphics[width=6cm,height=5cm]{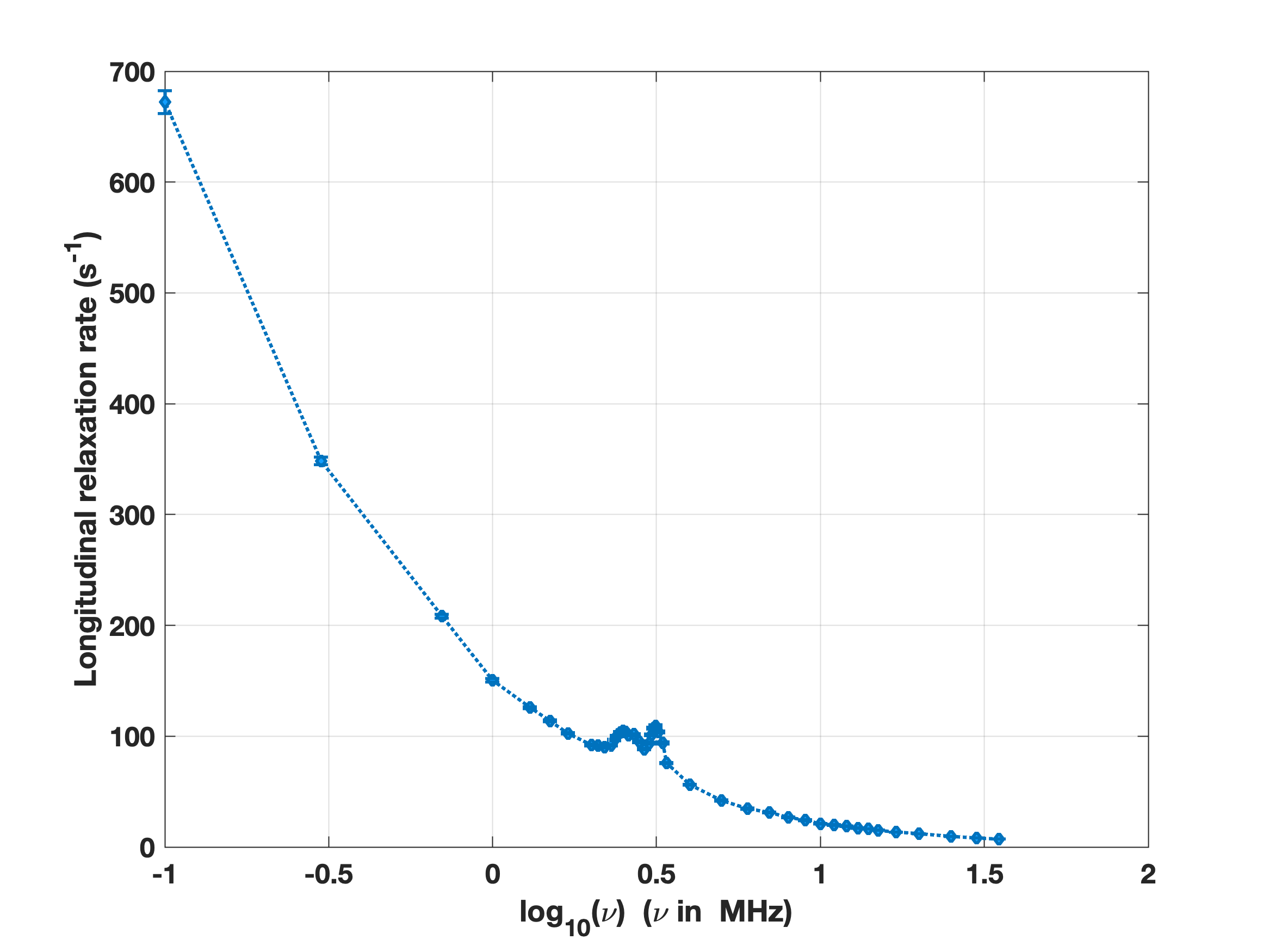}\\
(a) \hspace{5cm} (b)
\caption{NMRD profiles. (a) Parmigiano Reggiano sample.  (b) Dry nanosponge sample.}
\label{fig:NMRD_samp}
\end{figure}
\begin{figure}[h!]
\centering
\includegraphics[width=6cm,height=5cm]{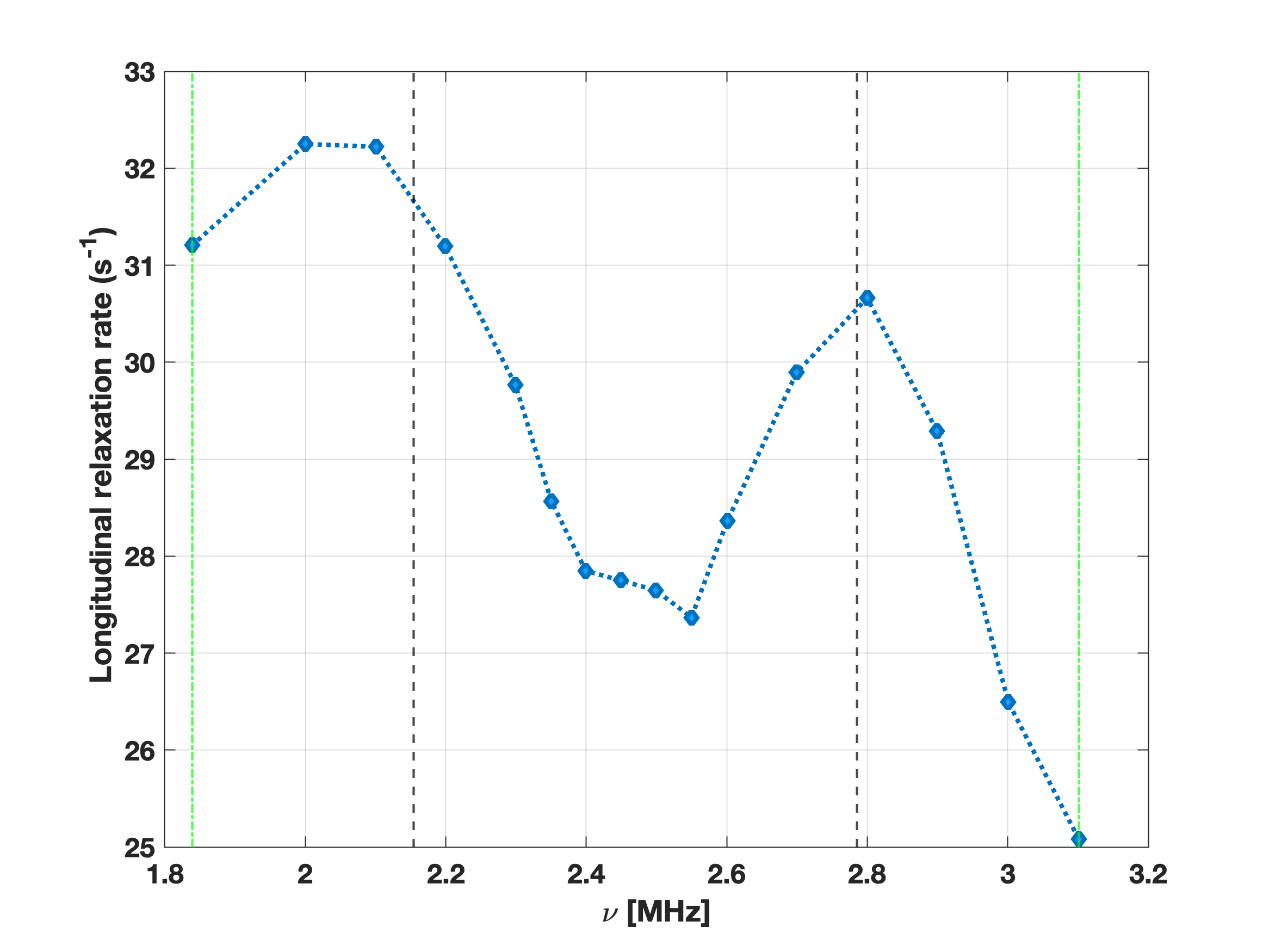}
\includegraphics[width=6cm,height=5cm]{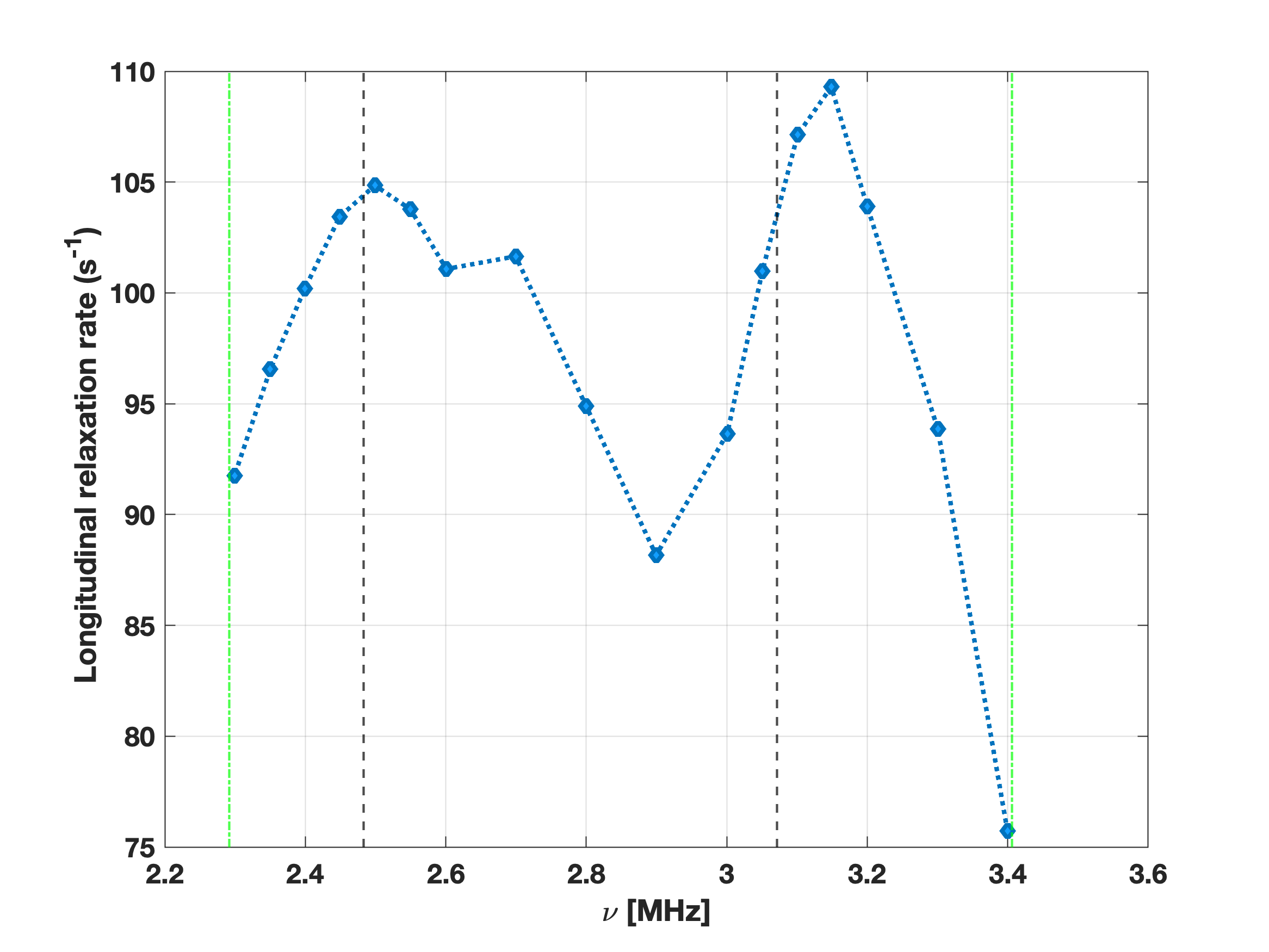}\\
(a) \hspace{5cm} (b)
\caption{Zoom of  quadrupolar dips. (a) Parmigiano Reggiano Cheese.  (b) Dry nanosponge sample.}
\label{fig:R1_zoom}
\end{figure}

The proposed  \Name  \ method has been used to compute the model parameters reported in table \ref{tab:Par_data}.
\begin{table}[h!]
\centering
\begin{tabular}{c| cc}
\hline
\hline
 & \multicolumn{2}{c}{Parameter values}\\
\hline
 & PR  & DN \\ 
\hline 
$R_0$ & 3.23 & 2.73 \\ 
$C^{NH}$ & 5.66 & 69.00 \\ 
$\Theta$ & 1.25 & 0.91 \\ 
$\Phi$ & 0.86 & 0.87 \\ 
$\tau_Q$ & 1.02 & 0.74 \\ 
$\nu_-$ & 2.1 & 2.56 \\ 
$\nu_+$ & 2.8 & 3.17 \\ 
\hline
\hline
MSE & $7.8887  \ 10^{-2}$ & $2.7853$\\
\hline
\end{tabular}
\caption{Values of the parameters fitted by \Name \ and MSE in the last row.}
\label{tab:Par_data}
\end{table}
The obtained correlation distributions  are represented  in figure \ref{fig:Ampl} in dark green line.
\begin{figure}[h!]
\centering
\includegraphics[width=6cm,height=5cm]{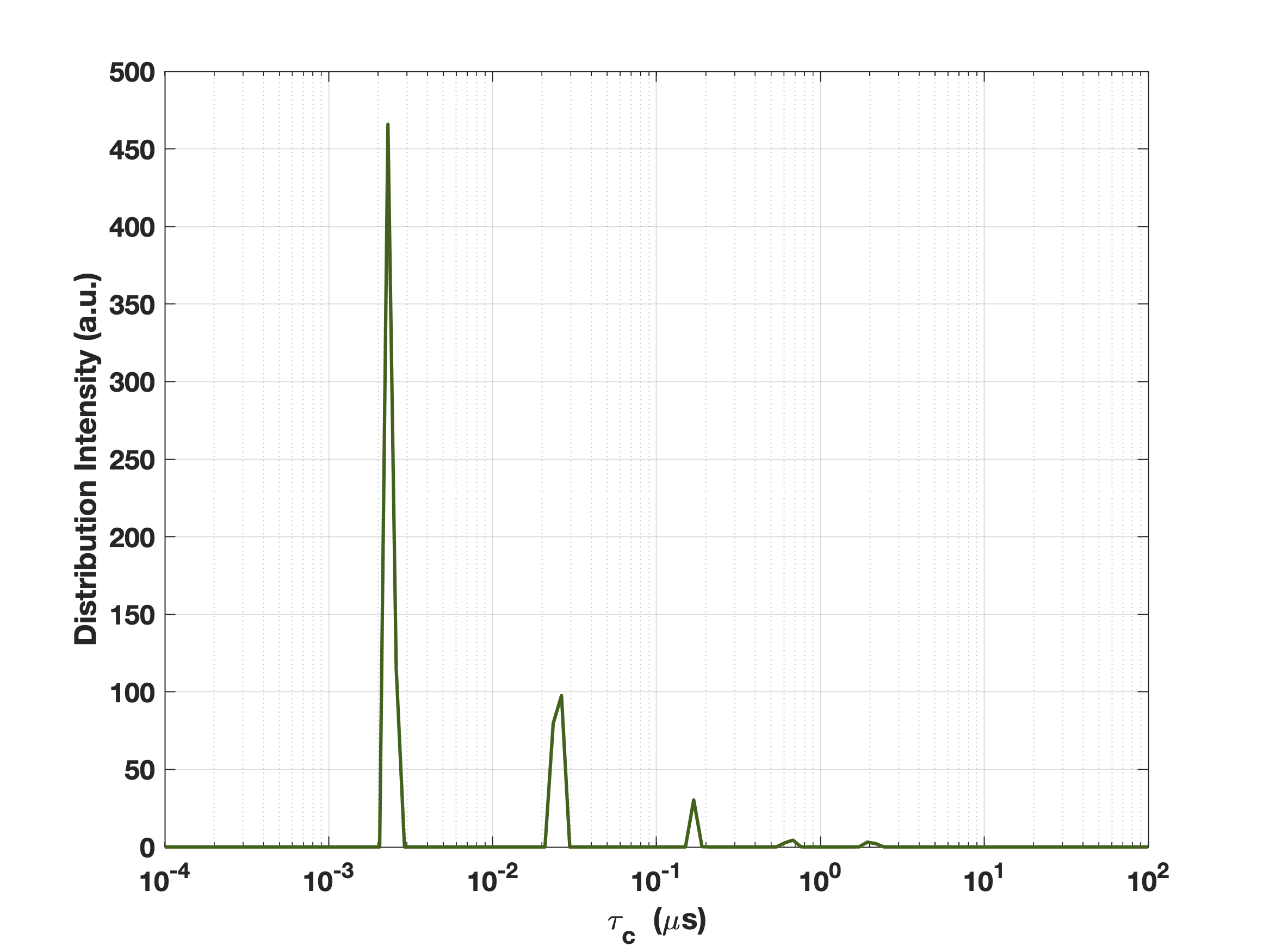}
\includegraphics[width=6cm,height=5cm]{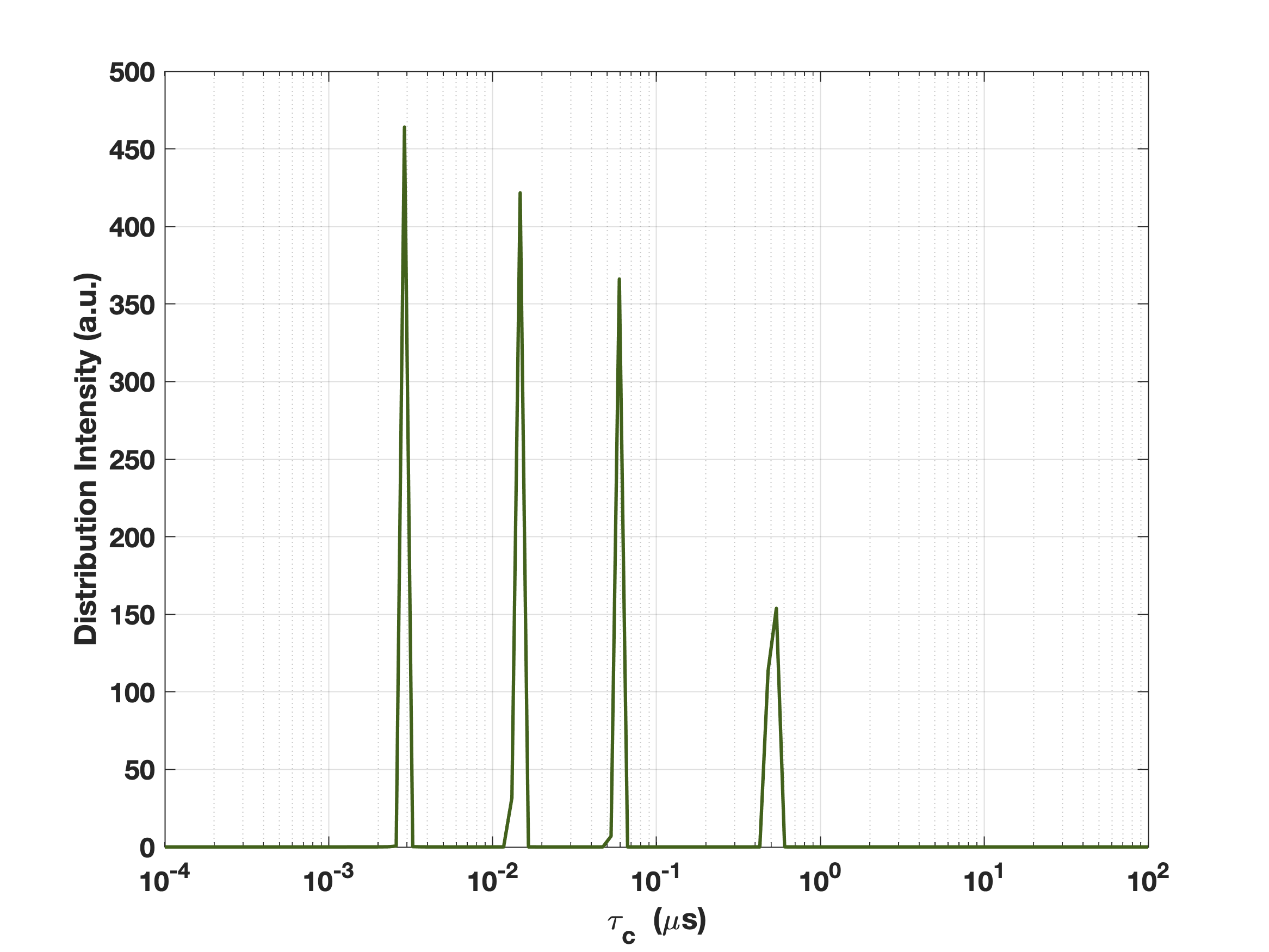}\\
(a) \hspace{5cm} (b)
\caption{Correlation distribution (dark green lines). (a) Parmigiano Reggiano sample.  (b) Dry nanosponge sample.
}
\label{fig:Ampl}
\end{figure}


Concerning the fit of the NMRD profiles we measured  the MSE reported in the last row of table  table \ref{tab:Par_data}.
The fitted NMRD profiles, represented in figure \ref{fig:fit_NMRD}, show in blue line the data and error bars while the fitted curves are represented in red  line for both samples.

\begin{figure}[h!]
\centering
\includegraphics[width=6cm,height=5cm]{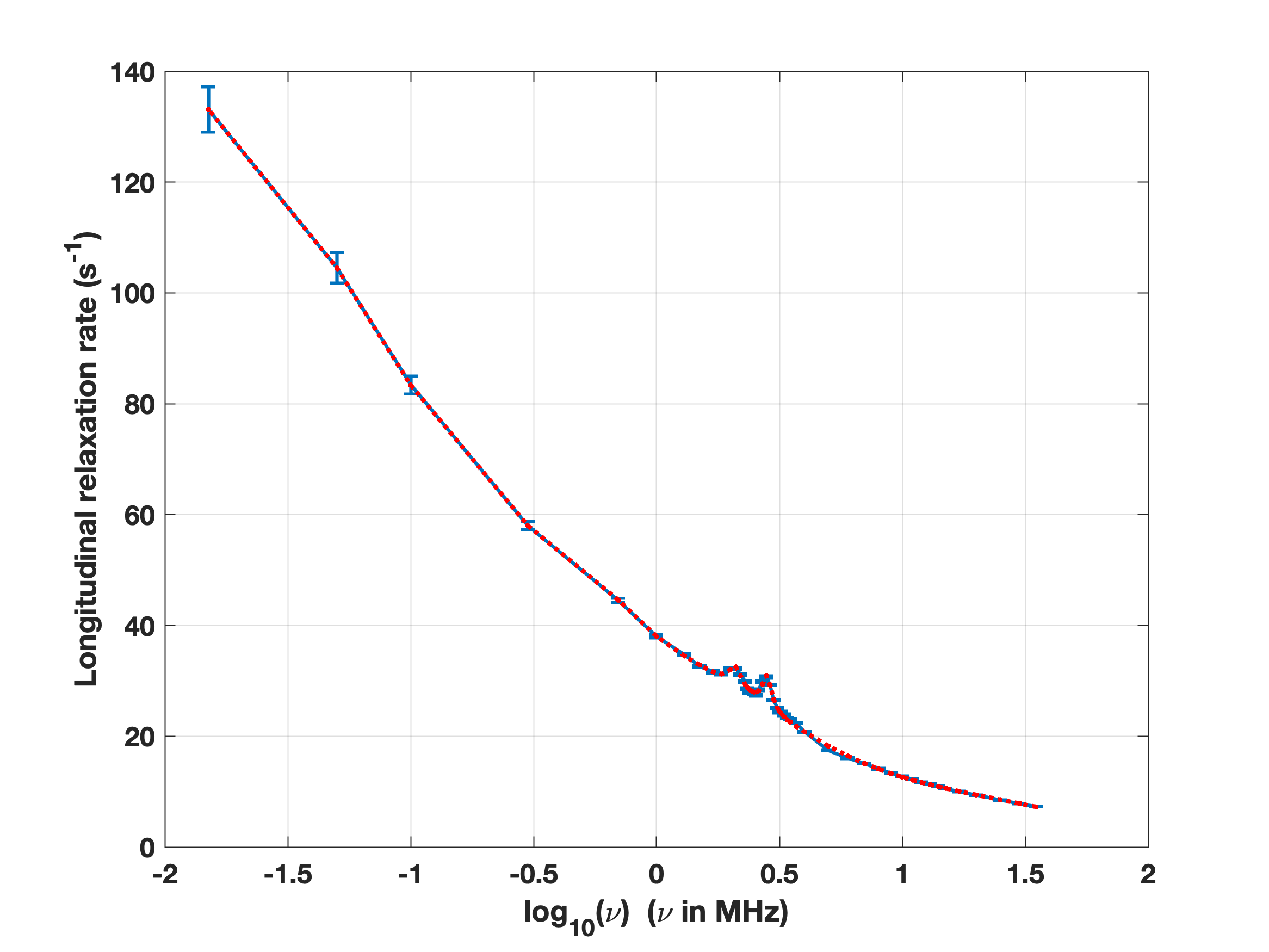}
\includegraphics[width=6cm,height=5cm]{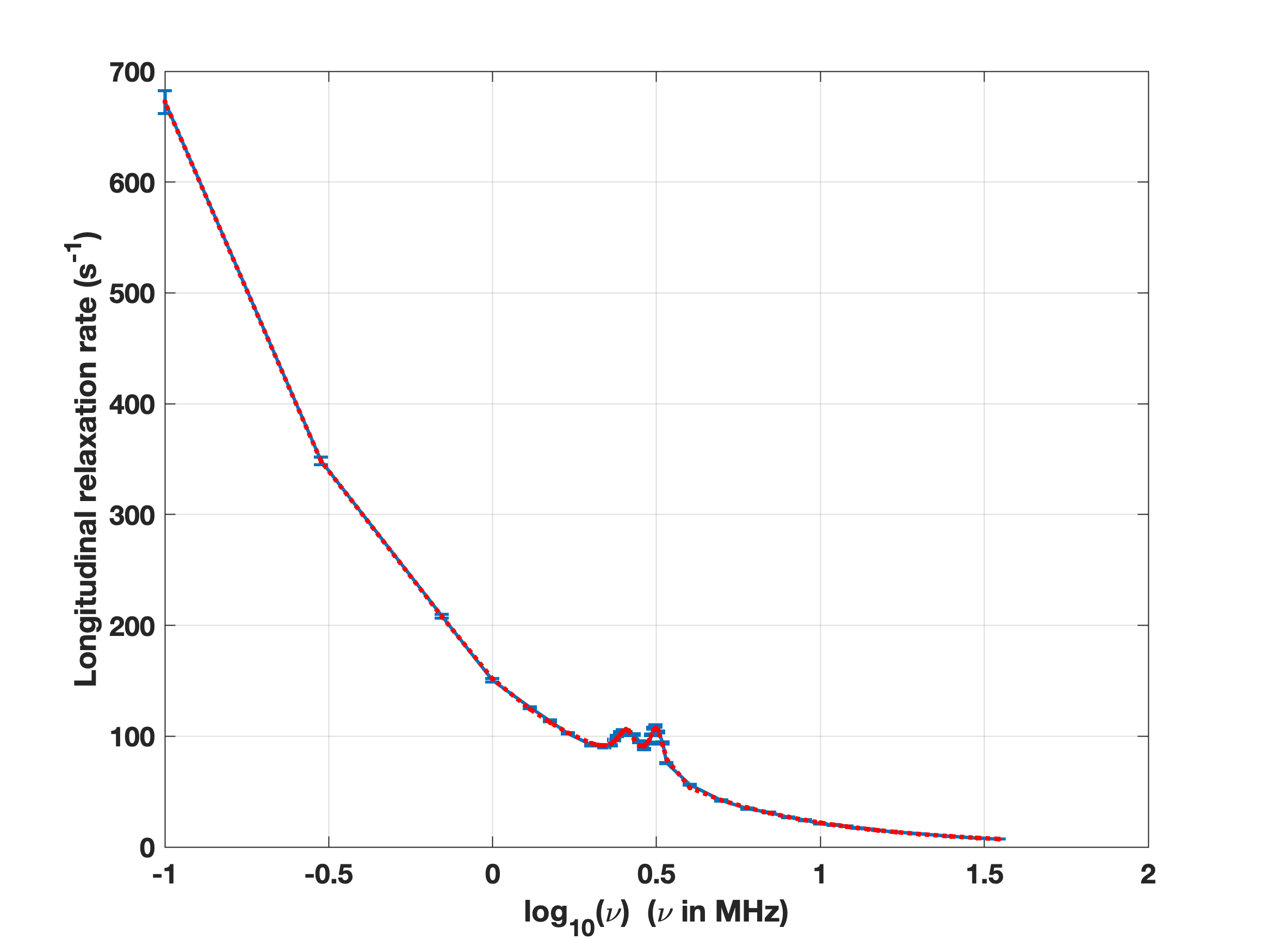}\\
(a) \hspace{5cm} (b)
\caption{NMRD data and error bars (blue lines) and fitted curve (red lines). (a) Parmigiano Reggiano sample.  (b) Dry nanosponge sample.
}
\label{fig:fit_NMRD}
\end{figure}
The zoom in the frequencies of QRE interval is shown in figure \ref{fig:zoom_fit}.
\begin{figure}[h!]
\centering
\includegraphics[width=6cm,height=5cm]{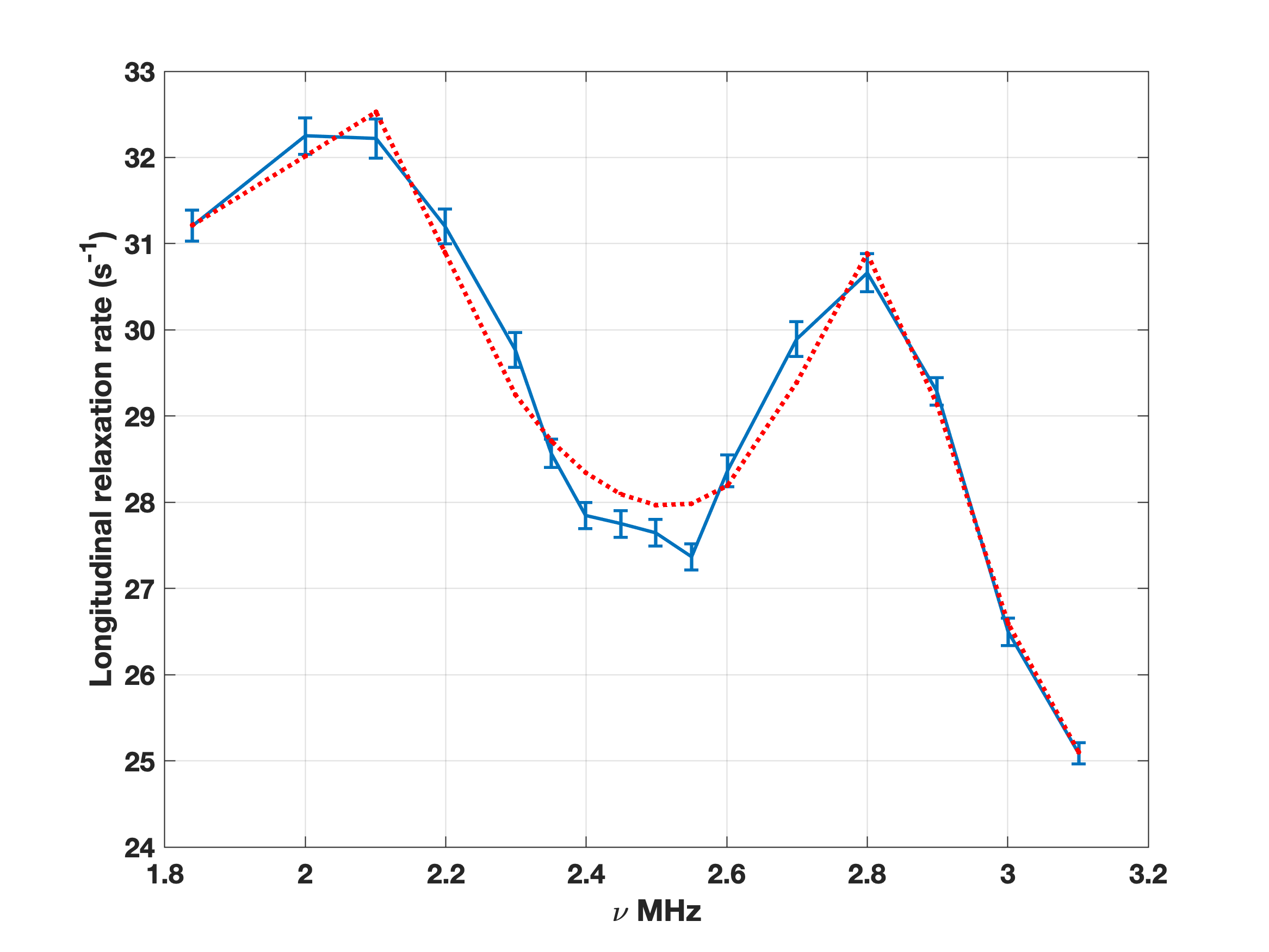}
\includegraphics[width=6cm,height=5cm]{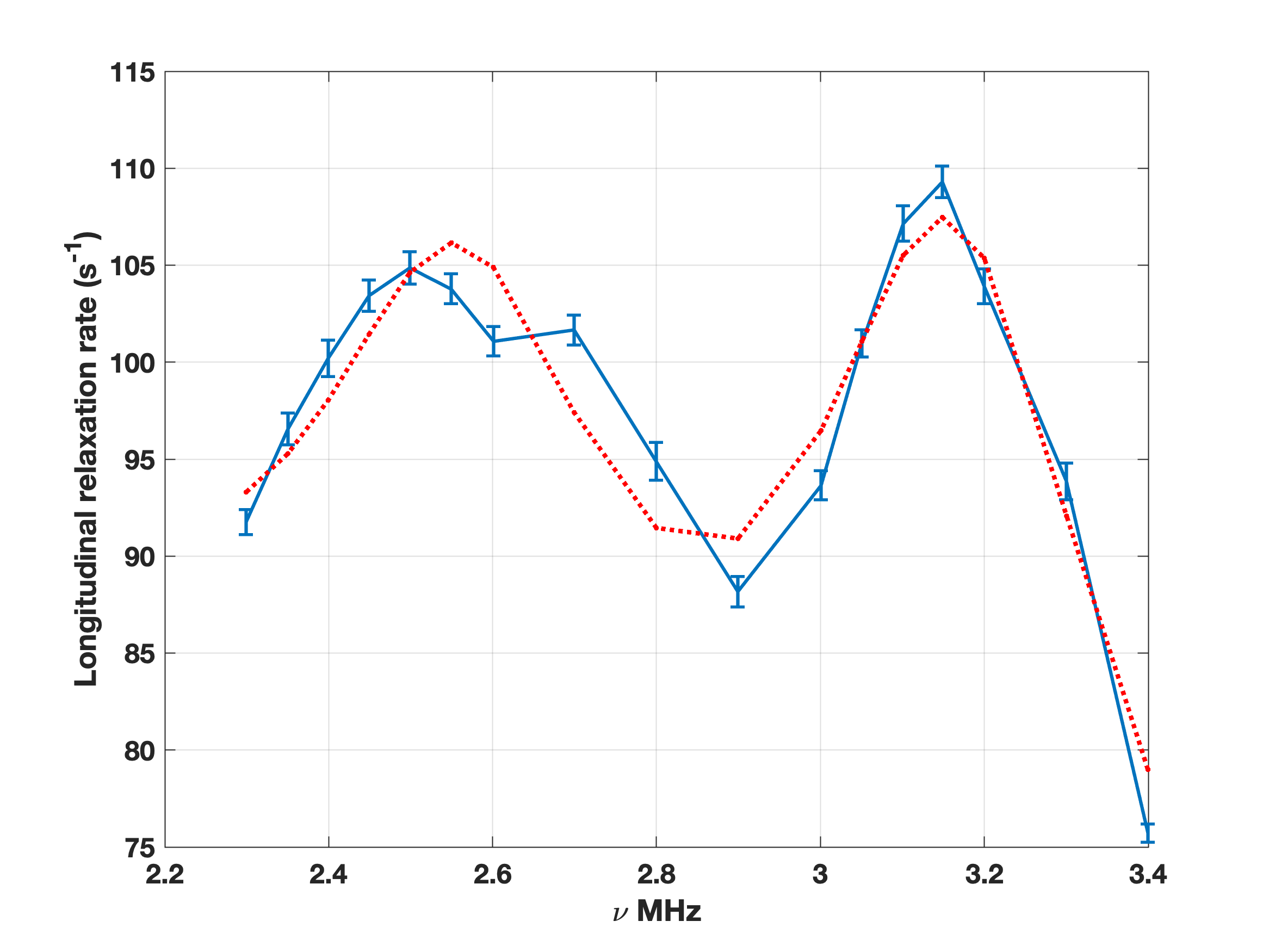}\\
(a) \hspace{5cm} (b)
\caption{Zoom of data and fitted curves in the QRE intervals.  NMRD data and error bars (blue lines) and fitted curve (red lines). (a) PR sample.  (b) DN  sample.
}
\label{fig:zoom_fit}
\end{figure}
The results confirm the excellent fit to the NMRD profile  (figure \ref{fig:fit_NMRD}) also in the QRE interval  (figure \ref{fig:zoom_fit}).

%

%
%
\section{Conclusion} \label{conclusions}
The present contribution investigates an automatic approach for analyzing the NMRD profiles in the presence of the quadrupolar relaxation enhancement. This feature yields a non-linear model whose parameters require the solution of a constrained non-linear least squares problem. Coupling the model-free approach and $L_1$ regularization, we tackle the constrained problem by a two-blocks non-linear Gauss-Seidel method.
We assess the well-posedness of the optimization problem (existence of a minimum) and the convergence of the GS iterations to a critical point.
Finally, we introduce an automatic convergent update rule of the regularization parameter based on the Balancing Principle.

The proposed algorithm is investigated both with synthetic and real data
and the results state that it is a robust, fast approach to obtain accurate
estimates of the correlation times distributions as well as modeling the quadrupolar function.

Moreover, we highlight that \Name \  can be viewed as a reference framework to construct parameter estimation procedures  
when the model parameters can be split into independent blocks  allowing the use 
of different computational approaches for each block.
In this regard, future work will include the extension of such a framework to different models of NMRD profiles where the number of correlation times $\tau$ in \eqref{eq:RHH} is assigned, and their values are to be estimated together with the corresponding component $\f(\tau)$.\\
Given the very accurate and promising results, \Name \ will be  included in the   Matlab software tool  {\tt FreeModelFFC Tool} for the inversion of NMRD profiles with QRE (available in  \url{https://site.unibo.it/softwaredicam/en/software}).


%
%


%
\section*{Acknowledgement}
G. Landi and F.  Zama were  supported by the Istituto Nazionale di Alta Matematica,  Gruppo Nazionale per il Calcolo Scientifico (INdAM-GNCS).



%
%
%
\end{document}